\documentclass[12pt,leqno,oneside]{amsart}
\usepackage{mathrsfs,dsfont}
\usepackage{amsmath,amstext,amsthm,amssymb,amscd}
\usepackage{charter}
\usepackage{typearea}
\usepackage{pdfsync}
\usepackage{mathtools}

\usepackage
{hyperref}
\usepackage[width=6.4in,height=8.5in]{geometry}

\numberwithin{equation}{section}

\newtheorem{Theorem}{Theorem}[section]
\newtheorem{Proposition}[Theorem]{Proposition}
\newtheorem{Lemma}[Theorem]{Lemma}
\newtheorem{Corollary}[Theorem]{Corollary}
\theoremstyle{definition}
\newtheorem{Definition}[Theorem]{Definition}

\newtheorem{Remark}[Theorem]{Remark}

\newtheorem{Problem}[Theorem]{Problem}

\newcommand{\db}{\overline\partial}

\newcommand{\wi}{\widetilde}
\DeclareMathOperator{\ric}{Ric}
\DeclareMathOperator{\codim}{codim}

\DeclareMathOperator{\dist}{dist}
\DeclareMathOperator{\Bs}{Bs}

\DeclareMathOperator{\ord}{ord}
\newcommand{\cali}[1]{\mathscr{#1}}
\newcommand{\cO}{\cali{O}} 

\newcommand{\cM}{\cali{M}}

\newcommand{\cC}{\cali{C}}

\newcommand{\cD}{\cali{D}}

\newcommand{\field}[1]{\mathbb{#1}}
\newcommand{\Z}{\field{Z}}
\newcommand{\R}{\field{R}}
\newcommand{\C}{\field{C}}
\newcommand{\N}{\field{N}}

\newcommand{\Q}{\field{Q}}
\newcommand{\B}{\field{B}}
\renewcommand{\P}{\field{P}}

\newcommand{\X}{\field{X}}

\newcommand{\reg}{\mathrm{reg} }
\newcommand{\sing}{\mathrm{sing} }
\newcommand{\FS}{{\rm FS}}
\newcommand{\eq}{\mathrm{eq} }
\newcommand{\req}{\mathrm{req} }

\newcommand{\PSH}{{\rm PSH}}
\newcommand{\Amp}{{\rm Amp}}
\newcommand{\NAmp}{{\rm NAmp}}

\newcommand{\Leb}{\mathop{\mathrm{Leb}}\nolimits}
\newcommand{\exph}{{\rm exph}}
\newcommand{\const}{\mathop{\mathrm{const}}\nolimits}

\newcommand{\ddbar}{{\partial\overline\partial}}

\newcommand{\comment}[1]{}
\newcommand{\ke}{\nobreak\hspace{.06em plus .03em}}
\newcommand*{\LargerCdot}{\raisebox{-0.25ex}
{\scalebox{1.5}{$\cdot$}}}

\begin{document}

\title[Holomorphic sections vanishing along subvarieties]
{Holomorphic sections of line bundles vanishing along subvarieties}

\author{Dan Coman}
\thanks{D.\ Coman is partially supported by the NSF Grant DMS-1700011}
\address{Department of Mathematics, Syracuse University, 
Syracuse, NY 13244-1150, USA}\email{dcoman@syr.edu}
\author{George Marinescu}
\address{Universit{\"a}t zu K{\"o}ln, Mathematisches Institut, Weyertal 86-90, 50931 K{\"o}ln, 
Deutschland    \newline
    \mbox{\quad}\,Institute of Mathematics `Simion Stoilow', Romanian Academy, Bucharest, Romania}
\email{gmarines@math.uni-koeln.de}
\thanks{G.\ Marinescu is partially supported by DFG funded project SFB/TRR 191}
\author{Vi{\^e}t-Anh Nguy{\^e}n}
\address{Universit\'e de Lille 1, 
Laboratoire de math\'ematiques Paul Painlev\'e, 
CNRS U.M.R. 8524,  \newline
    \mbox{\quad}\,59655 Villeneuve d'Ascq Cedex, 
France}
\email{Viet-Anh.Nguyen@math.univ-lille1.fr}
\thanks{V.-A. Nguyen is partially supported by 
Vietnam Institute for Advanced Study in Mathematics (VIASM)}
\thanks{Funded through the Institutional Strategy of the 
University of Cologne within the German Excellence Initiative}
\subjclass[2010]{Primary 32L10; Secondary 32A60, 32U40, 32W20, 53C55, 81Q50}
\keywords{(Partial) Bergman kernel function, singular Hermitian 
metric, big line bundle, big  cohomology class,
   holomorphic sections}
\date{August 29, 2019}

\pagestyle{myheadings}

\begin{abstract}
Let $X$ be a compact normal complex space of dimension $n$, 
and $L$ be a holomorphic line bundle on $X$. 
Suppose $\Sigma=(\Sigma_1,\ldots,\Sigma_\ell)$ is an $\ell$-tuple 
of distinct irreducible proper analytic subsets of $X$, 
$\tau=(\tau_1,\ldots,\tau_\ell)$ is an $\ell$-tuple of positive real numbers, 
and consider the space $H^0_0 (X, L^p)$ of global holomorphic sections of 
$L^p:=L^{\otimes p}$ that vanish to order at least $\tau_{j}p$ along 
$\Sigma_{j}$, $1\leq j\leq\ell$. We find necessary and sufficient conditions 
which ensure that $\dim H^0_0(X,L^p)\sim p^n$, analogous to
Ji-Shiffman's criterion for big line bundles. We give estimates of the 
partial Bergman kernel, investigate the
convergence of the Fubini-Study currents and their potentials, and the  
equilibrium distribution of normalized currents of integration along zero divisors 
of random holomorphic sections in $H^0_0 (X, L^p)$ as $p\to\infty$.
Regularity results for the equilibrium envelope are also included.
\end{abstract}

\maketitle
\tableofcontents

\section{Introduction} \label{introduction}

Let $(X,L)$ be a polarized projective manifold of dimension $n$,
let $\Sigma$ be a complex hypersurface of $X$, and let
$\tau$ be a positive real number.
The study of holomorphic sections of $L^p$ which vanish
to order at least $p\tau$ along $\Sigma$ received much attention in the past few years.
The density function of this space, called partial Bergman kernel, 
appears in a natural  way, in several contexts, especially in K\"ahler
geometry and pluripotential theory (linked to the notion
of extremal quasiplurisubharmonic functions with poles along 
$\Sigma$) see e.{\ke}g.\ \cite{Ber07,RS13, PS14, RWN14,RWN17,CM17,ZeZh}.
One of the motivations is the notion of slope of the hypersurface $\Sigma$ in the sense of
Ross-Thomas \cite{RT06} and its relation to the existence of a constant scalar curvature
K\"ahler metric in $c_1(L)$.

In this paper we consider a compact normal complex space $X$ of dimension $n$,
a holomorphic line bundle $L$ over $X$
and the space $H^0_0(X,L^p)$ of holomorphic sections
vanishing to order at least $p\tau_j$ along irreducible proper
analytic subsets $\Sigma_j\subset X$, $j=1,\ldots,\ell$.
We study algebraic and analytic objects associated to $H^0_0(X,L^p)$,
especially the partial Bergman kernels, the Fubini-Study currents and their potentials.

We first give an analytic characterization for $H^0_0(X,L^p)$ to be big,
which means by definition that $\dim H^0_0(X,L^p)\sim p^n$, $p\to\infty$.
This criterion, stated in terms of singular Hermitian metrics with
positive curvature current in the spirit of the Ji-Shiffman/Bonavero/Ta-kayama criterion
for big line bundles, involves a desingularization of $X$ where
the $\Sigma_j$'s become divisors.

Next we prove that under natural hypotheses the Fubini-Study
currents associated to $H^0_0(X,L^p)$ and their potentials converge
as $p\to\infty$. The limit of the sequence of Fubini-Study potentials is the 
push-forward $\varphi_\eq$ of a certain equilibrium envelope   
with logarithmic poles defined on a desingularization.
The sequence of the Fubini-Study currents
converge to the corresponding equilibrium current $T_\eq$.
These are analogues of Tian's theorem \cite{Ti90} which applies for smooth
Hermitian metrics with positive curvature. In the context
of singular Hermitian metrics they were introduced in \cite{CM15,CM13}.
The convergence of the Fubini-Study currents/potentials is based
on the asymptotics of the logarithm of the partial Bergman kernel
(see also \cite{CM15,CM13,CMM17,DMM16} for results of this type concerning
the full Bergman kernel).

Returning to the case of a polarized projective manifold $(X,L)$,
Shiffman-Zelditch \cite{SZ99} showed how Tian's theorem
can be applied to obtain the distribution of the zeros of random
holomorphic sections of $H^0(X,L^p)$.
Dinh-Sibony \cite{DS06} used meromorphic transforms 
to obtain an estimate of the speed of convergence of zeros
to the equilibrium distribution (see also \cite{DMS} for the non-compact setting).
Random polynomials or more generally holomorphic sections 
in high tensor powers of a holomorphic line
bundle and the distribution of
their zeros represent a classical subject in analysis 
(see e.{\ke}g.\ \cite{BlPo31,ET50,Ham56,Kac49}).
The result of \cite{SZ99} was generalized for singular metrics
whose curvature is a K\"ahler current in \cite{CM15} and for sequences 
of line bundles over normal complex spaces in \cite{CMM17} (see
also \cite{CM13, DMM16}).
We show here that the equilibrium distribution of random zeros of sections
from $H^0_0(X,L^p)$ is the equilibrium current $T_\eq$ and give an
estimate of the convergence speed.

%

\subsection{Background and notation}
Let $X$ be a compact normal complex space of dimension $n$. 
If $L$ is a holomorphic line bundle on $X$, we let $L^p := L^{\otimes p}$ 
and denote by $H^0(X, L^p)$ the space of global holomorphic sections of $L^p$. 
Since $X$ is compact, the spaces $H^0(X, L^p)$ are finite dimensional. 
Given $S\in H^0(X, L^p)$, we denote by $[S = 0]$ the current of integration 
(with multiplicities) over the analytic hypersurface $\{S=0\}\subset  X$. 
If $h$ is a singular Hermitian metric on $L$ we denote by $c_1(L,h)$ 
its {\it curvature current}. 

Suppose now that $X$ is a compact complex manifold. 
For a closed current $T$ of bidegree $(1,1)$ on $X$, let $\{T\}$ 
denote its class in the Dolbeault cohomology group $H^{1,1}(X)$. 
If $L$ is a holomorphic line bundle over $X$ we denote by $c_1(L)$ 
its first Chern class in $H^{1,1}(X,\Z)$. We have that $\{c_1(L,h)\}=c_1(L)$, 
for any singular Hermitian metric $h$ on $L$. 
The line bundle $L$ is called \emph{big} if its Kodaira-Iitaka dimension 
equals the dimension of $X$ (see \cite[Definition 2.2.5]{MM07}). 
One has that $L$ is big if and only if 
$\limsup_{p\to\infty} p^{-n}\dim  H^0 (X, L^p )>0$ (see \cite[Theorem 2.2.7]{MM07}). 
By the Ji-Shiffman/Bonavero/Takayama criterion \cite[Theorem 2.3.30]{MM07}, 
$L$ is big if and only if it admits a strictly positively curved singular Hermitian metric $h$ 
(see Section \ref{SS:Preliminaries} for definitions). 

Throughout the article, we denote by $\lfloor r\rfloor$ the greatest integer $\leq r\in\R$, and we let $d^c:= \frac{1}{2\pi i}\,(\partial -\overline\partial)$, so $dd^c=\frac{i}{\pi}\,\partial\overline\partial$.

\subsection{Sections vanishing along subvarieties}\label{SS:big}

\par We consider in this paper the following general setting:

\smallskip

(A) $X$ is a compact, irreducible, normal (reduced) complex space of dimension $n$, 
$X_\reg$ denotes the set of regular points of $X$, and $X_\sing$ denotes the set 
of singular points of $X$.

\smallskip

(B) $L$ is a holomorphic line bundle on $X$.

\smallskip

(C) $\Sigma=(\Sigma_1,\ldots,\Sigma_\ell)$ is an $\ell$-tuple of distinct irreducible proper analytic subsets of $X$ such that $\Sigma_j\not\subset X_\sing$, for every $j\in\{1,\ldots,\ell\}$.

\smallskip

(D) $\tau=(\tau_1,\ldots,\tau_\ell)$ is an $\ell$-tuple of positive real numbers such that $\tau_j>\tau_k$, for every $j,k\in\{1,\ldots,\ell\}$ with $\Sigma_j\subset\Sigma_k$.

\smallskip

For $p\geq1$ let $H^0_0 (X, L^p)$ be the space of sections $S\in H^0(X,L^p)$ that vanish to order at least $\tau_{j}p$ along $\Sigma_{j}$, for all $1\leq j\leq\ell$. More precisely, let 
\begin{equation}\label{e:tjp}
t_{j,p}=\begin{cases}\tau_jp  &\text{if $\tau_jp\in\N$} \\
\lfloor\tau_jp\rfloor+1 &\text{if $\tau_jp\not\in\N$}\end{cases}\;,\;\;1\leq j\leq\ell\,,\,\;p\geq1\,.
\end{equation}
Then 
\begin{equation}\label{e:H00}
H^0_0(X, L^p)=H^0_0(X,L^p,\Sigma,\tau):=\{S\in H^0(X, L^p):\,\ord(S,\Sigma_j)\geq t_{j,p},\;1\leq j\leq\ell\}\,,
\end{equation}
where $\ord(S,Z)$ denotes the vanishing order of $S$ along 
an irreducible analytic subset $Z$ of $X$, $Z\not\subset X_\sing$. 
 
 \begin{Definition}\label{D:big} 
 We say that  the triplet $(L,\Sigma,\tau )$ is {\it big} if 
 $\displaystyle\limsup_{p\to\infty} \frac{\dim H^0_0 (X, L^p)}{p^n}>0$. 
 \end{Definition} 
The first problem we address in this article is the following:

%
%

\begin{Problem}\label{prob1}
Characterize the big triplets $(L,\Sigma,\tau)$.
\end{Problem}

We first give an answer to Problem \ref{prob1} in the case when $X$ is a complex manifold 
and $\Sigma_j$ are irreducible hypersurfaces in $X$. 
Namely, we have the following analog of the Ji-Shiffman's
criterion for big line bundles \cite[Theorem 4.6]{JS93}, see
also \cite{Bon93}, \cite[Theorem 2.3.30]{MM07} :

\begin{Theorem}\label{T:main0} 
Let $X,L,\Sigma,\tau$ verify assumptions (A)-(D), 
and suppose that $X$ is smooth and $\dim\Sigma_j=n-1$ for all $j=1,\ldots,\ell$. 
The following are equivalent: 

(i) $(L,\Sigma,\tau)$ is big;

(ii) There exists a singular Hermitian metric $h$ on $L$ such that 
$c_1(L,h)-\sum_{j=1}^\ell\tau_j[\Sigma_j]$ is a K\"ahler current on $X$;

(iii) There exist $p_0\in\N$ and $c>0$ such that $\dim H^0_0(X,L^p)\geq cp^n$ for all $p\geq p_0$.
\end{Theorem}

Here $[\Sigma_j]$ denotes the current of integration along $\Sigma_j$. 
Recall that a K\"ahler current is a positive closed current $T$ of bidegree $(1,1)$ 
such that $T\geq\varepsilon\omega$ for some number $\varepsilon>0$ 
and some Hermitian form $\omega$ on $X$. 
To find a solution to Problem \ref{prob1} in the general case, 
we first use Hironaka's theorem on resolution of singularities to prove the following result.

\begin{Proposition}\label{P:divisorization} 
Let $X$ and $\Sigma$ verify assumptions (A) and (C). 
Then there exist a compact complex manifold $\wi X$ of dimension $n$ 
and a surjective holomorphic map $\pi:\wi X\to X$, 
given as the composition of finitely many blow-ups with smooth center, 
with the following properties:

(i) There exists an analytic subset $Y$ of $X$ such that 
$\dim Y\leq n-2$, $Y\subset X_\sing\cup\bigcup_{j=1}^\ell\Sigma_j$, 
$X_\sing\subset Y$, $\Sigma_j\subset Y$ if $\dim\Sigma_j\leq n-2$, 
$E=\pi^{-1}(Y)$ is a divisor in $\wi X$ that has only normal crossings, 
and $\pi:\wi X\setminus E\to X\setminus Y$ is a biholomorphism. 

(ii) There exist (connected) smooth complex hypersurfaces 
$\wi\Sigma_1,\ldots,\wi\Sigma_\ell$ in $\wi X$, which have only normal crossings, 
such that $\pi(\wi\Sigma_j)=\Sigma_j$. Moreover, if $\dim\Sigma_j=n-1$ then 
$\wi\Sigma_j$ is the final strict transform of $\Sigma_j$, 
and if $\dim\Sigma_j\leq n-2$ then $\wi\Sigma_j$ is an irreducible component of $E$.

(iii) If $F\to X$ is a holomorphic line bundle and 
$S\in H^0(X,F)$ then $\ord(S,\Sigma_j)=\ord(\pi^\star S,\wi\Sigma_j)$, 
for all $j=1,\ldots,\ell$.
\end{Proposition}
  
\begin{Definition}\label{D:divisorization} 
If $\wi X$, $\pi$, $\wi\Sigma:=(\wi\Sigma_1,\ldots,\wi\Sigma_\ell)$, 
verify the conclusions of Proposition \ref{P:divisorization}, 
we say that $(\wi X,\pi,\wi\Sigma)$ is a divisorization of $(X,\Sigma)$. 
\end{Definition} 

Divisorizations are not unique. Note that if $X$ is a manifold and 
$\Sigma_1,\ldots,\Sigma_\ell$ are smooth hypersurfaces with 
simple normal crossings, then $(X,{\rm Id},\Sigma)$ is a divisorization 
of $(X,\Sigma)$, where ${\rm Id}$ is  the identity map. 
We now give an answer to Problem 1 in the general case:

\begin{Theorem}\label{T:main1} 
Let $X,L,\Sigma,\tau$ verify assumptions (A)-(D). The following are equivalent: 

(i) $(L,\Sigma,\tau)$ is big;

(ii) For every divisorization $(\wi X,\pi,\wi\Sigma)$ of $(X,\Sigma)$, 
there exists a singular Hermitian metric $h^\star$ on $\pi^\star L$ 
such that $c_1(\pi^\star L,h^\star)-\sum_{j=1}^\ell\tau_j[\wi \Sigma_j]$ 
is a K\"ahler current on $\wi X$; 

(iii) There exist a divisorization $(\wi X,\pi,\wi\Sigma)$ of $(X,\Sigma)$ 
and a singular Hermitian metric $h^\star$ on $\pi^\star L$ 
such that $c_1(\pi^\star L,h^\star)-\sum_{j=1}^\ell\tau_j[\wi \Sigma_j]$ 
is a K\"ahler current on $\wi X$;

(iv) There exist $p_0\in\N$ and $c>0$ such that 
$\dim H^0_0(X,L^p)\geq cp^n$ for all $p\geq p_0$.
\end{Theorem}

An interesting consequence of Theorem \ref{T:main1} is the following. 
Assume that $(L,\Sigma,\tau)$ is big and all $\Sigma_j$ have dimension $n-1$. 
If one fixes proper analytic subsets $\Sigma_j'\subset\Sigma_j$ and considers 
the subspace $V_p\subset H^0_0(X,L^p)$ of sections that vanish to the higher order 
$(\tau_j+\delta)p$ along $\Sigma_j'$, then it holds as well that $\dim V_p\gtrsim p^n$ 
for all $p$ large enough, provided that $\delta>0$ is sufficiently small 
(see Corollary \ref{C:big} for the precise statement). 

Proposition \ref{P:divisorization}, Theorem \ref{T:main0}, 
and Theorem \ref{T:main1} are proved in Section \ref{S:big}.

\subsection{Equidistribution of zeros}\label{SS:equidist}
 
 Let $X,L,\Sigma,\tau$ verify assumptions (A)-(D), and assume in addition that there exists a K\"ahler form $\omega$ on $X$ and that $h$ is a singular Hermitian metric on $L$. We fix a smooth Hermitian metric $h_0$ on $L$ and write 
\begin{equation}\label{e:varphi}
\alpha:=c_1(L,h_0)\,,\,\;h=h_0e^{-2\varphi}\,,
\end{equation}
 where $\varphi\in L^1(X,\omega^n)$ is called the 
 {\it (global) weight of  $h$  relative to $h_0$}. The metric $h$ is called bounded, 
 continuous, resp.\ H\"older continuous, if $\varphi$ is a bounded, 
 continuous, resp.\ H\"older continuous, function on $X$. 

Let $H^0_{(2)}(X,L^p)$ be the Bergman space of $L^2$-holomorphic sections 
in $L^p$ relative to the metric $h^p:=h^{\otimes p}$ and the volume form 
$\omega^n$ on $X$, endowed with the inner product
\[(S,S')_p:=\int_{X}\langle S,S'\rangle_{h^p}\,\frac{\omega^n}{n!}\,,\]
and set $\|S\|_{p}^2:=(S,S)_{p}$. Let 
\[H^0_{0,(2)}(X,L^p)=H^0_{0,(2)}(X,L^p,\Sigma,\tau,h^p,\omega^n):=
H^0_{(2)}(X,L^p)\cap H^0_0(X,L^p)\] 
be the Bergman subspace of $L^2$-holomorphic sections in $H^0_0(X,L^p)$, 
where $H^0_0(X,L^p)$ was defined in \eqref{e:H00}. 
We assume in the sequel that the metric $h$ is bounded, so 
\[H^0_{(2)}(X,L^p)=H^0(X,L^p)\,,\,\;H^0_{0,(2)}(X,L^p)=H^0_0 (X, L^p)\,.\] 
For every $p\geq 1$ we consider the projective space
\[ \X_p:=\P H^0_{0,(2)}(X,L^p)\,,\,\;d_p:=\dim\X_p=\dim H^0_{0,(2)}(X,L^p)-1\,,\]
equipped with the Fubini-Study volume $\sigma_p=\omega_\FS^{d_p}\,$, 
where by $\omega_\FS$ we denote the Fubini-Study K\"ahler form on a projective space 
$\P^N$. We also consider the probability space
\[(\X_\infty,\sigma_\infty):= \prod_{p=1}^\infty (\X_p,\sigma_p)\,.\]  
The second problem we address in this article is the following:

%
%

\begin{Problem}\label{prob2}
Assume that $(L,\Sigma,\tau )$ is big and 
the metric $h$ is bounded. Do zeros of sequences from
$(\X_\infty,\sigma_\infty)$ equidistribute towards 
a positive closed current $T$ of bidegree $(1,1)$? 
That is, for $\sigma_\infty$-{\ke}a.{\ke}e.\ $\{s_p\}_{p\geq1}\in\X_\infty$, 
we have $\displaystyle\frac{1}{p}\,[s_p=0]\to T$ as $p\to\infty$, 
in the weak sense of currents on $X$. If yes, express $T$ in terms of $h$ and 
estimate the speed of convergence.
\end{Problem}

The bigness of $(L,\Sigma,\tau )$ is a reasonable assumption, 
in order to ensure that the spaces $H^0_{0}(X,L^p)$ have sufficiently many sections. 
Let $P_p,\gamma_p$ be the Bergman kernel function and Fubini-Study current of 
$H^0_{0,(2)}(X,L^p)$, defined in \eqref{e:Bkf} and \eqref{e:FS}. Then  
\begin{equation}\label{e:FSpot}
\frac{1}{p}\,\gamma_p=c_1(L,h)+\frac{1}{2p}\,dd^c\log P_p=
\alpha+dd^c\varphi_p\,,\,\text{ where }\,\varphi_p=
\varphi+\frac{1}{2p}\,\log P_p\,.
\end{equation}
We call the function $\varphi_p$ the {\em global Fubini-Study potential} 
of $\gamma_p$. To answer Problem \ref{prob2}, we first study the convergence 
of the Fubini-Study currents. We have the following:

\begin{Theorem}\label{T:FSpot}
Let $X,L,\Sigma,\tau$ verify assumptions (A)-(D), and assume that 
$(L,\Sigma,\tau)$ is big and there exists a K\"ahler form $\omega$ on $X$. 
Let $h$ be a continuous Hermitian metric on $L$ and $\alpha,\varphi_p$ be 
defined in \eqref{e:varphi}, resp.\ \eqref{e:FSpot}. Then there exists an 
$\alpha$-psh function $\varphi_\eq$ on $X$ such that as $p\to\infty$,
\begin{equation}\label{e:cFSpot}
\varphi_p\to\varphi_\eq \text{ in } L^1(X,\omega^n),\;
\frac{1}{p}\,\gamma_p=\alpha+dd^c\varphi_p\to T_\eq:=
\alpha+dd^c\varphi_\eq \text{ weakly on $X$\,.}
\end{equation}
Moreover, if $h$ is H\"older continuous then there exist a constant 
$C>0$ and $p_0\in\N$ such that 
\begin{equation}\label{e:eFSpot}
\int_X|\varphi_p-\varphi_\eq|\,\omega^n\leq C\,\frac{\log p}{p}\,,
\,\;\text{for all $p\geq p_0$}\,.
\end{equation}
\end{Theorem}

\begin{Definition}\label{D:T_eq}
The current $T_\eq$ from Theorem \ref{T:FSpot} is called 
{\it the equilibrium current associated to $(L,h,\Sigma,\tau )$}.
\end{Definition}

Theorem \ref{T:FSpot} is proved in Section \ref{S:FSpotentials}. 
The function $\varphi_\eq$ is constructed as follows. 
Let $(\wi X,\pi,\wi\Sigma)$ be a divisorization of $(X,\Sigma)$ as in 
Definition \ref{D:divisorization}, and let $\wi\alpha=\pi^\star\alpha$, 
$\wi\varphi=\varphi\circ\pi$. We introduce in Section \ref{S:env} 
the equilibrium envelope $\wi\varphi_\eq$ of $(\wi\alpha,\wi\Sigma,\tau,\wi\varphi)$, 
as the largest $\wi\alpha$-psh function dominated by $\wi\varphi$ on $\wi X$, 
and with logarithmic poles of order $\tau_j$ along $\wi\Sigma_j$, $1\leq j\leq\ell$ 
(see \eqref{e:enveq}, \eqref{e:wienveq}). In Theorem \ref{T:reg} 
we study the regularity of $\wi\varphi_\eq$ when $\wi\varphi$ is continuous, 
resp.\ H\"older continuous, and show that $\wi\varphi_\eq$ is continuous outside 
a certain analytic subset of $\wi X$, resp.\ H\"older with singularities along 
that analytic subset (see Definition \ref{D:Hosing}). The function $\varphi_\eq$ 
is then constructed by pushing down $\wi\varphi_\eq$ to $X$. 

Theorem \ref{T:FSpot} is a generalization of the following foundational
result of Tian \cite{Ti90} (with improvements by \cite{Ca99,Ru98,Z98}, 
see also \cite[Theorem 5.1.4]{MM07}): 
If $X$ is a compact K\"ahler manifold and $(L,h)\to X$
is a positive line bundle (with smooth metric $h$), 
then $\varphi_p\to\varphi$ and 
$\frac{1}{p}\,\gamma_p\to c_1(L,h)$ as $p\to\infty$ in the $\cC^\infty$-topology. 
If $h$ is a singular metric whose curvature is a K\"ahler current
it was shown in \cite[Theorem 5.1]{CM15} that $\varphi_p\to\varphi$
in $L^1(X,\omega^n)$
and $\frac{1}{p}\,\gamma_p\to c_1(L,h)$ weakly as $p\to\infty$.
On the other hand, Bloom \cite{Bl05,Bl09} (cf.\ also Bloom-Levenberg
\cite{BL15}) pointed out the role of the extremal plurisubharmonic
functions in equidistribution theory for polynomials and Berman \cite{Ber07,Ber09} 
extended this point of view to the context of K\"ahler manifolds.
In \cite[Theorem 1.3]{DMM16} it is shown that in the case of a polarized
projective manifold $(X,L)$ and for a H\"older continuous weight $\varphi$,
we have $\|\varphi_p-\varphi_\eq\|_{\infty}=O(p^{-1}\log p)$ as $p\to\infty$.
We also note that the statement of Theorem \ref{T:FSpot} is new
even in the case when $X$ is smooth and $\Sigma=\emptyset$ (see Corollary \ref{cor:FSpot}).
 
Using Theorem \ref{T:FSpot}, we obtain a positive answer to the above 
equidistribution problem in the case when the metric $h$ is continuous.
In this formulation it can be seen as a large deviation principle in this context.

\begin{Theorem}\label{T:main2}
Let $X,L,\Sigma,\tau$ verify assumptions (A)-(D), 
let $h$ be a singular Hermitian metric on $L$, and assume that 
$(L,\Sigma,\tau)$ is big and there exists a K\"ahler form $\omega$ on $X$.
 
(i) If $h$ is continuous then $\displaystyle\frac{1}{p}\,[s_p=0]\to T_\eq$ 
as $p\to\infty$, in the weak sense of currents on $X$, 
for $\sigma_\infty$-{\ke}a.{\ke}e.\ $\{s_p\}_{p\geq1}\in\X_\infty$\,.
 
(ii) If $h$ is H\"older continuous then there exists a constant $c>0$ 
with the following property: For any sequence of positive numbers 
$\{\lambda_p\}_{p\geq1}$ such that 
\[\liminf_{p\to\infty} \frac{\lambda_p}{\log p}>(1+n)c\,,\]
there exist subsets $E_p\subset\X_p$ such that, for all $p$ sufficiently large, 

(a) $\sigma_p(E_{p})\leq cp^n\exp(-\lambda_p/c)$\,,

(b) if $s_p\in\X_p\setminus E_p$ we have 
\[\Big|\Big \langle\frac{1}{p}\,[s_p=0]-T_\eq,\phi\Big\rangle\Big|
\leq\frac{c\lambda_p}{p}\,\| \phi\|_{\cC^2}\,,\] 
for any $(n-1,n-1)$-form $\phi$ of class $\cC^2$  on $X$. 

In particular, the last estimate holds for $\sigma_\infty$-{\ke}a.{\ke}e.\ 
$\{s_p\}_{p\geq1}\in\X_\infty$ provided that $p$ is large enough.
\end{Theorem} 

The proof of Theorem \ref{T:main2} is given in Section \ref{S:Tmain2}.
We refer to \cite{BCM,Bl05,Bl09,BL15,CM15,CM13,CM17,CMM17,CMN18,DMM16,DMS,DS06,
SZ99}
and to the surveys \cite{BCHM,Z18} 
for equidistribution results for holomorphic sections in various contexts.


 \section{Preliminaries}\label{S:Preliminaries}
 
We start by recalling a few notions of pluripotential theory on analytic 
spaces that will be needed throughout the paper. 
We then recall some basic facts about Bergman kernels and Fubini-Study currents. 
 
 \subsection{Compact complex manifolds and analytic spaces}\label{SS:Preliminaries}
 
 Let $X$ be a compact complex manifold and let $\omega$ 
 be a Hermitian form on $X$. If $T$ is a positive closed current on $X$ 
 we denote by $\nu(T,x)$ the Lelong number of $T$ at $x\in X$ (see e.{\ke}g.\ \cite{D93}). 
 A function $\varphi:\ X\to \R\cup\{-\infty\}$ is called {\it quasiplurisubharmonic} (qpsh) 
 if it is locally the sum of a psh function and smooth one. 
 Let $\alpha$ be a smooth real closed $(1,1)$-form on $X.$  
 A qpsh function $\varphi$ is called {\it $\alpha$-plurisubharmonic} 
 ($\alpha$-psh) if $\alpha+dd^c\varphi\geq 0$ in the sense of currents. 
 We denote by $\PSH(X,\alpha)$ the set of all $\alpha$-psh functions on $X$. 
 The Lelong number of an $\alpha$-psh function $\varphi$ 
 at a point $x\in X$ is defined by $\nu(\varphi,x):=\nu(\alpha+dd^c\varphi,x)$. 
 Note that if $\varphi=u+\chi$ 
 near $x$, where $u$ is psh and $\chi$ is smooth, then $\nu(\varphi,x)=\nu(u,x)$. 

 Since in general the $\ddbar$-lemma does not hold on $X$, 
 we will consider the $\ddbar$-cohomology and particularly the 
 space $H^{1,1}_\ddbar(X,\R)$ (see e.{\ke}g.\ \cite{Bo04}). 
 This space is finite dimensional, and if $\alpha$ is a smooth 
 real closed $(1,1)$-form on $X$ we denote its $\ddbar$-cohomology class 
 by $\{\alpha\}_\ddbar$. Note that if $X$ is a compact K\"ahler manifold then 
 by the $\ddbar$-lemma  $H^{1,1}_\ddbar(X,\R)=H^{1,1}(X,\R)$ 
 and we write $\{\alpha\}_\ddbar=\{\alpha\}$.
 
 \begin{Definition}
 A real closed current $T$ of bidegree $(1,1)$ on $X$ is called a 
 {\it  K\"ahler current} if $T\geq \varepsilon\omega$ for some number 
 $\varepsilon>0$. A class $\{\alpha\}_\ddbar$ is called {\it big} if it contains a K\"ahler current.  
\end{Definition}

Suppose that $\{\alpha\}_\ddbar$ is big. By Demailly's regularization theorem \cite{D92}, 
one can find a K\"ahler current $T\in\{\alpha\}_\ddbar$ with {\em analytic singularities}, 
i.{\ke}e.\ of the form $T=\alpha+dd^c\varphi\geq\epsilon\omega$, where $\epsilon>0$ 
and $\varphi$ is a qpsh function such that
\[\varphi=c\log\Big(\sum_{j=1}^N |g_j|^2\Big)+\chi\,,\]
locally on $X$, where $c>0$, $\chi$ is smooth and $g_j$ are holomorphic functions, 
$1\leq j\leq N$.

The {\em non-ample locus} of $\{\alpha\}_\ddbar$ is defined in 
\cite[Definition\ 3.16]{Bo04} as the set 
\begin{align*}
\NAmp\big(\{\alpha\}_\ddbar\big)&=
\bigcap\big\{E_+(T):\,\text{$T\in\{\alpha\}_\ddbar$ K\"ahler current}\big\} \\
&=\bigcap\big\{E_+(T):\,\text{$T\in\{\alpha\}_\ddbar$ 
K\"ahler current with analytic singularities}\big\},
\end{align*}
where $E_+(T)=\{x\in X:\,\nu(T,x)>0\}$, and the second equality 
follows by Demailly's regularization theorem \cite{D92}. 
Hence $\NAmp\big(\{\alpha\}_\ddbar\big)$ is an analytic subset of $X$. 
The {\it ample locus} of $\alpha$ is 
$\Amp\big(\{\alpha\}_\ddbar\big):=X\setminus\NAmp\big(\{\alpha\}_\ddbar\big)$. 
It is shown in \cite[Theorem\ 3.17]{Bo04} that there exists a K\"ahler current 
$T\in\{\alpha\}_\ddbar$ with analytic singularities such that 
$E_+(T)=\NAmp\big(\{\alpha\}_\ddbar\big)$. 
  
\medskip 

Let now $X$ be a complex space. A chart $(U,\iota,V)$ on X is a triple consisting 
of an open set $U \subset X$, a closed complex space $V \subset G \subset \C^N$ 
in an open set $G$ of $\C^N$ and a biholomorphic map $\iota:\ U \to V$ 
(in the category of complex spaces). The map $\iota:\ U \to G \subset \C^N$ 
is called a local embedding of $X$. We write $X=X_\reg \cup X_\sing$, 
where $X_\reg$ and $X_\sing$ are the sets of regular and singular points of $X$. 
Recall that a reduced complex space $(X,\cO)$ is called {\it normal} 
if for every $x \in X$ the local ring $\cO_x$ is integrally closed in its quotient field 
$\cM_x$, cf.\ \cite[p.\ 124]{GR84}). 
Every normal complex space is locally irreducible and locally pure-dimensional 
(see \cite[p.\ 125]{GR84}), and $X_\sing$ is a closed complex subspace of $X$ with 
$\codim X_\sing \geq 2$.  

A continuous (resp.\ smooth) function on $X$ is a function $\varphi:\ X \to\ \C$ 
such that for every $x \in X$ there exists a local embedding $\iota:\ U\to G \subset \C^N$ 
with $x \in U$ and a continuous (resp.\ smooth) function $\tilde\varphi:G \to\C$ such that 
$\varphi|_U =\tilde\varphi\circ\iota$. A (strictly) plurisubharmonic (psh) function on $X$ 
is a function $\varphi:X\to[-\infty, \infty)$ such that for every $x \in X$ 
there exists a local embedding $\iota: U\to G \subset \C^N$ with $x \in U$ 
and a (strictly) psh function $\tilde\varphi: G\to[-\infty,\infty)$ such that 
$\varphi|_U= \tilde\varphi\circ\iota$.  If $\tilde\varphi$ can be chosen continuous 
(resp.\ smooth), then $\varphi$ is called a continuous (resp.\ smooth) psh function. 
We let $\PSH(X)$ denote the set of all psh functions on $X$. 

Assume now that $X$ has pure dimension $n$. We consider currents on $ X$ as
defined in \cite{D85}: If $\cD^{p,q}(X)$ is the space of forms with compact support, 
endowed with the inductive limit topology, then the dual $\cD_{p,q}(X)$ of 
$\cD^{p,q}( X)$ is the space of currents of bidimension $(p, q)$, or bidegree 
$(n-p, n-q)$, on $X$. If $T\in\cD_{n-1,n-1} (X)$ is so that, for every $x\in X$, 
there is a domain $U$ containing $x$ and  $v \in \PSH(U)$ with $T=dd^c v$ on $U$, 
then $T$ is positive and closed, and we say that $v$ is a local potential of $T$.  
A Hermitian form on $X$ is a smooth $(1, 1)$-form $\omega$ such that for every point 
$x \in X$ there exist a local embedding $\iota:\ U\ni x \to G \subset\C^N$ 
and a Hermitian form $\tilde\omega$ on $G$ with $\omega=\iota^\star\tilde\omega$ 
on $U \cap X_\reg$. Note that $\omega^n/n!$ gives locally an area measure on $X$. 
A K\"ahler form on $X$ is a current $T\in \cD_{n-1,n-1}(X)$ whose local potentials 
extend to smooth strictly psh functions in local embeddings of $X$ to Euclidean spaces. 
We call $X$ a K\"ahler space if $X$ admits a K\"ahler form 
(see also \cite[p.\ 346]{Gr62}, \cite [Section 5]{Ohs87}).
  
The notions of qpsh and $\alpha$-psh function on $X$, where $\alpha$ 
is a smooth real closed $(1,1)$-form on $X$, are defined exactly as in the 
case when $X$ is smooth. We denote by $\PSH(X,\alpha)$ the set of all 
$\alpha$-psh functions on $X$. If $X$ is compact, a function $\rho:X\to\R$ 
is called H\"older continuous if there exists a finite open cover of $X$ by charts 
$(U,\iota , V)$, $V\subset G\subset\C^N$, such that $\rho\mid_U$ 
is H\"older continuous with respect to the metric on $U$ 
induced by the Euclidean distance on $\C^N$. 

If $(L,h)$ is a singular Hermitian holomorphic line bundle over $X$, 
the {\em curvature current} $c_1(L,h)$ of $h$ is defined as in the 
case when $X$ is smooth \cite{D90}. If $e_U$ is a local holomorphic frame 
of $L$ on some open set $U\subset X$ then $|e_U|_h=e^{-\varphi_U}$, 
where $\varphi_U\in L^1_{loc}(U)$ is called the {\it local weight} 
of the metric $h$ with respect to $e_U$, and $c_1(L,h)\mid_U=dd^c\varphi_U$. 
We say that $h$ is {\it positively curved}, resp.\ {\it strictly positively curved}, 
if $c_1(L,h)\geq 0$, resp.\ $c_1(L,h)\geq\varepsilon\omega$ for some 
$\varepsilon>0$ and some Hermitian form  $\omega$ on $X$.

\subsection{Bergman kernel functions and Fubini-Study currents}\label{SS:BK-FS}
 Let $X$ be as in (A), $\omega$ be a Hermitian form on $X$, and $(L,h)$ be a singular Hermitian holomorphic line bundle on $X$. Since $X$ is compact, the space $H^0(X,L)$ is finite dimensional. Let $H^0_{ (2)}(X, L) = H^0_{ (2)}(X,L,h, \omega^n)$ be the Bergman space of $L^2$-holomorphic sections of $L$ relative to the metric $h$ and the volume form $\omega^n /n!$ on $X$, endowed with the inner product
\begin{equation}\label{e:inner_product}
(S,S'):=\int_X\langle S,S'\rangle_h\,\frac{\omega^n}{n!}\,.
\end{equation}
Set $\|S\|^2=\|S\|^2_{h,\omega^n}:=(S,S)$. 

Let $V$ be a subspace of $H^0_{ (2)}(X, L)$, $r=\dim V$, and $S_1,\ldots,S_r$ be an orthonormal basis of $V$. The {\em Bergman kernel function} $P=P_V$ of $V$ is defined by 
\begin{equation}\label{e:Bkf}
P(x)=\sum_{j=1}^r|S_j(x)|_h^2,\;\;|S_j(x)|_h^2:=\langle S_j(x),S(x)\rangle_h,\;x\in X.
\end{equation}
Note that this definition is independent of the choice of basis. Let $U$ be an open set in $X$ such that $L$ has a local holomorphic frame $e_U$ on $U$. Then $|e_U|_h=e^{-\varphi_U}$, where $\varphi_U\in L^1_{loc}(U,\omega^n)$, and $S_j=s_je_U$, where $s_j\in\cO_X(U)$. It follows that 
\begin{equation}\label{e:Bk_local}
\log P\mid_U= \log\Big(\sum_{j=1}^r|s_j|^2 \Big)-2\varphi_U\,,
\end{equation}
which shows that $\log P_p\in L^1(X,\omega^n)$.

The Kodaira map determined by $V$ is the meromorphic map given by 
\begin{equation}\label{e:Kodaira_alg_dual}
\Phi=\Phi_V:X\dashrightarrow \P(V^\star)\,,\,\;\Phi(x)=\{S\in V:\,S(x)=0\},\;x\in X\setminus \Bs(V)\,,
\end{equation}
where a point in  $\P(V^\star)$ is identified with a hyperplane through the origin in $V$ and $\Bs(V)= \{x \in X :\,S(x) = 0,\,\forall\,S\in V\}$ is the base locus of $V$. We define the {\em Fubini-Study current}  $\gamma=\gamma_V$ of $V$ by
\begin{equation}\label{e:FS}
\gamma:= \Phi^\star(\omega_\FS),
\end{equation}
where $\omega_\FS$ denotes the Fubini-Study form on $\P(V^\star)$. Then $\gamma$ is a positive closed current of bidegree $(1,1)$ on $X$, and if $U$ is as above we have
\begin{equation}\label{e:FS_local}
\gamma\mid_U = \frac{1}{2}dd^c \log\Big(\sum_{j=1}^r |s_j|^2\Big).
\end{equation}
Hence by \eqref{e:Bk_local},
\begin{equation}\label{e:Bk_FS}
\gamma= c_1 (L, h)+ \frac{1}{2}\,dd^c \log P\,.
\end{equation}

\medskip

Let now $X,L,\Sigma,\tau$ verify assumptions (A)-(D) and $H^0_0 (X, L^p)$ be the space defined in \eqref{e:H00}. If $h$ is a bounded metric on $L$ then $H^0_0 (X, L^p )\subset H^0_{(2)}(X,L^p,h^p,\omega^n)$. The Bergman kernel function $P_p$ of $H^0_0 (X, L^p)$ is called the {\em partial Bergman kernel function} of the space of sections that vanish to order $\tau p$ along $\Sigma$. It satisfies the following variational principle:
\begin{equation}\label{e:variational-carac}
P_p (x)= \max \left\lbrace |S(x)|^2_{h^p} :\,S\in H^0_0(X, L^p ),\,\|S\|_p = 1\right\rbrace,
\end{equation}
where $\|\LargerCdot\|_p$ denotes the norm given by the inner product in $H^0_{(2)}(X,L^p,h^p,\omega^n)$.

 
\section{Dimension growth of spaces of sections vanishing along subvarieties}\label{S:big}
 
In this section we give the proofs of Theorem \ref{T:main0}, Proposition \ref{P:divisorization}, and Theorem \ref{T:main1}.
\subsection{Divisorization}\label{SS:divisorization} We start by proving the existence of the divisorization of $(X,\Sigma)$ claimed in Proposition \ref{P:divisorization}. We will use the following theorems of Hironaka on resolution of singularities. For the first one we refer the reader to \cite[Theorem 13.2]{BM97}.

\begin{Theorem}[Hironaka]\label{T:H1} 
If $X$ is a compact, reduced complex space then there exists a compact complex manifold $\widehat X$ and a surjective holomorphic map $\sigma:\widehat X\to X$ such that $\sigma:\widehat X\setminus E\to X_\reg$ is a biholomorphism, where $E=\sigma^{-1}(X_\sing)$ is a divisor with only normal crossings. Moreover, if $X$ is irreducible then $\widehat X$ is connected and $\dim\widehat X=\dim X$. 
\end{Theorem}

The second one is Hironaka's embedded resolution of singularities theorem 
(see e.{\ke}g.\  \cite[Theorems 10.7 and 1.6]{BM97}, \cite[Theorem 2.1.13]{MM07}). 

\begin{Theorem}[Hironaka]\label{T:H2} 
Let $X$ be a complex manifold of dimension $n$, and $A\subset X$ be a compact analytic subset of $X$. Then there exist a complex manifold $\widetilde X$ and a surjective holomorphic map $\sigma:\wi X\to X$, given as the composition of finitely many blow-ups with smooth center, with the following properties:

(i) $E=\sigma^{-1}(A_\sing)$ is a divisor in $\wi X$, and $\sigma:\wi X\setminus E\to X\setminus A_\sing$ is a biholomorphism.

(ii) The strict transform $A'=\overline{\sigma^{-1}(A_\reg)}$ is smooth and $A'$, $E$ simultaneously have only normal crossings.
\end{Theorem}

If $A=A_1\cup\ldots\cup A_m$ has irreducible components $A_j$, and $A_j'$ is the strict transform of $A_j$, it follows from Theorem \ref{T:H2} that $A_j'$ are pairwise disjoint connected submanifolds of $\wi X$. Performing blow-ups of $\wi X$ with centers $A_j'$, for all $j$ with $\dim A_j\leq n-2$, one obtains the following version of Theorem \ref{T:H2} (see also \cite[Theorem 2.1]{CM15}):

\begin{Theorem}\label{T:H3} 
Let $X$ be a complex manifold of dimension $n$, and $A\subset X$ be a compact analytic subset of $X$ with irreducible components $A_1,\ldots,A_m$. Then there exist a complex manifold $\widetilde X$ and a surjective holomorphic map $\sigma:\wi X\to X$, given as the composition of finitely many blow-ups with smooth center, with the following properties:

(i) If $Y=A_\sing\cup\,\bigcup\{A_j:\,\dim A_j\leq n-2\}$ then $E=\sigma^{-1}(Y)$ is the final exceptional divisor, $\sigma:\wi X\setminus E\to X\setminus Y$ is a biholomorphism, and $E$ has only normal crossings.

(ii) There exist (connected) smooth complex hypersurfaces $\wi A_1,\ldots,\wi A_m$ in $\wi X$, which have only normal crossings, such that $\sigma(\wi A_j)=A_j$. Moreover, if $\dim A_j=n-1$ then $\wi A_j$ is the final strict transform of $A_j$, and  if $\dim A_j\leq n-2$ then $\wi A_j$ is an irreducible component of $E$.

(iii) If $F\to X$ is a holomorphic line bundle and $S\in H^0(X,F)$ then $\ord(S,A_j)=\ord(\sigma^\star S,\wi A_j)$, for all $j=1,\ldots,m$.
\end{Theorem}

\begin{proof} The existence of $\wi X$ and $\sigma$ with properties $(i)$-$(ii)$ follows directly from Theorem \ref{T:H2}, as previously described. Property $(iii)$ clearly holds for $j$ with $\dim A_j=n-1$, since $\wi A_j=\overline{\sigma^{-1}(A_j\setminus Y)}$ is the final strict transform of $A_j$ and $\sigma:\wi X\setminus E\to X\setminus Y$ is a biholomorphism. If $j$ is such that $\dim A_j\leq n-2$ and $A_j''\subset\wi X_1$ is the strict transform of $A_j$ produced in Theorem \ref{T:H2}, then $\wi A_j=\pi^{-1}(A_j'')$, where $\pi:\wi X\to\wi X_1$ is the blow-up of $\wi X_1$ with center $A_j''$. So property $(iii)$ follows easily from the local description of the blow-up map $\pi$.
\end{proof}

\begin{proof}[Proof of Proposition \ref{P:divisorization}]
Let $\widehat\sigma:\widehat X\to X$ be a desingularization of $X$ as in Theorem \ref{T:H1}, let $\widehat E=\widehat\sigma^{-1}(X_\sing)$ be the exceptional divisor, and $\widehat\Sigma_j$ be the strict transform of $\Sigma_j$. Hence $\widehat\sigma:\widehat X\setminus\widehat E\to X_\reg$ is a biholomorphism. Since $\Sigma_j\not\subset X_\sing$ we have that $\widehat\Sigma_j\neq\emptyset$ and $\dim\widehat\Sigma_j=\dim\Sigma_j$. Moreover $\widehat\Sigma_j$ is irreducible since $\Sigma_j$ is. Note that $\widehat\Sigma_j\subset\widehat\Sigma_k$ if and only if $\Sigma_j\subset\Sigma_k$.

We next apply Theorem \ref{T:H3} repeatedly, starting with $\widehat X,\widehat\Sigma_1,\ldots,\widehat\Sigma_\ell$, as follows. Let $\widehat\Sigma_{j_1},\ldots,\widehat\Sigma_{j_k}$ be the minimal elements of $\{\widehat\Sigma_1,\ldots,\widehat\Sigma_\ell\}$ with respect to inclusion. We apply Theorem \ref{T:H3} to $\widehat X$ and $A=\widehat\Sigma_{j_1}\cup\ldots\cup\widehat\Sigma_{j_k}$, to obtain a map $\sigma_1:\widehat X_1\to\widehat X$ verifying properties $(i)$-$(iii)$. Let $\widehat\Sigma_j'$ be the strict transform of $\widehat\Sigma_j$ by $\sigma_1$, for $j\in\{1,\ldots,\ell\}\setminus\{j_1,\ldots,j_k\}$, and note that $\widehat\Sigma_j'\subset\widehat\Sigma_k'$ if and only if $\widehat\Sigma_j\subset\widehat\Sigma_k$. We now apply Theorem \ref{T:H3} to $\widehat X_1$ and the analytic subset given by the union of the minimal elements of 
\[\big\{\widehat\Sigma_j':\,j\in\{1,\ldots,\ell\}\setminus\{j_1,\ldots,j_k\}\big\}\] 
with respect to inclusion, to obtain a map $\sigma_2:\widehat X_2\to\widehat X_1$ verifying properties $(i)$-$(iii)$. Repeating this procedure finitely many times we resolve all of the sets $\Sigma_j$. Finally, we apply Theorem \ref{T:H3} one more time in order to make the resulting smooth hypersurfaces $\wi\Sigma_1,\ldots,\wi\Sigma_\ell$ and the final exceptional divisor (including the preimage of $\widehat E$) simultaneously have only normal crossings. Taking the composition of the maps $\sigma_j$ we obtain a compact complex manifold $\wi X$ and a surjective holomorphic map $\sigma:\wi X\to\widehat X$, given as the composition of finitely many blow-ups with smooth center, with the following properties: 

$(a)$ There exists an analytic set $\widehat Y\subset\widehat E\cup\bigcup_{j=1}^\ell\widehat\Sigma_j$ such that $\dim\widehat Y\leq n-2$, $E_0:=\sigma^{-1}(\widehat Y)$ is a divisor in $\wi X$ with only normal crossings, and $\sigma:\wi X\setminus E_0\to \widehat X\setminus\widehat Y$ is a biholomorphism.

$(b)$ There exist smooth complex hypersurfaces $\wi\Sigma_j\subset\wi X$, $1\leq j\leq\ell$, which have only normal crossings, such that $\sigma(\wi\Sigma_j)=\widehat\Sigma_j$. Moreover, if $\dim\widehat\Sigma_j=n-1$ then $\wi\Sigma_j$ is the final strict transform of $\widehat\Sigma_j$ by $\sigma$, and  if $\dim\widehat\Sigma_j\leq n-2$ then $\wi\Sigma_j$ is an irreducible component of $E_0$. 

$(c)$ If $\widehat F\to\widehat X$ is a holomorphic line bundle and $S\in H^0(\widehat X,\widehat F)$ then $\ord(S,\widehat\Sigma_j)=\ord(\sigma^\star S,\wi\Sigma_j)$, for all $j=1,\ldots,\ell$.

\smallskip

We define $\pi:=\widehat\sigma\circ\sigma:\wi X\to X$, and set $Y:=X_\sing\cup\widehat\sigma(\widehat Y)$. Since $\widehat\sigma(\widehat Y)$ is an analytic subset of $X$ of dimension $\leq n-2$, we have $\dim Y\leq n-2$. 

By $(b)$, if $\dim\widehat\Sigma_j\leq n-2$ then $\wi\Sigma_j\subset E_0$, so $\widehat\Sigma_j=\sigma(\wi\Sigma_j)\subset\widehat Y$ and $\Sigma_j=\widehat\sigma(\widehat\Sigma_j)\subset\widehat\sigma(\widehat Y)\subset Y$. Since $\widehat Y\subset\widehat E\cup\bigcup_{j=1}^\ell\widehat\Sigma_j$, we have $\widehat\sigma(\widehat Y)\subset X_\sing\cup\bigcup_{j=1}^\ell\Sigma_j$, hence $Y\subset X_\sing\cup\,\bigcup_{j=1}^\ell\Sigma_j$. Moreover,
\[\widehat E\cup\widehat Y\subset\widehat\sigma^{-1}(X_\sing)\cup\widehat\sigma^{-1}(\widehat\sigma(\widehat Y))=\widehat\sigma^{-1}(Y)\,,\]
\[\widehat\sigma^{-1}(\widehat\sigma(\widehat Y))\subset\widehat E\cup\widehat\sigma^{-1}\big(\widehat\sigma(\widehat Y)\setminus X_\sing\big)=\widehat E\cup(\widehat Y\setminus\widehat E)=\widehat E\cup\widehat Y\,.\]
Thus $\widehat\sigma^{-1}(Y)=\widehat E\cup\widehat Y$. Let $\widehat E'$ be the strict transform of $\widehat E$ by $\sigma$. It is easy to see that $\sigma^{-1}\big(\widehat\sigma^{-1}(Y)\big)=\widehat E'\cup E_0$. Thus $E:=\widehat E'\cup E_0$ is a divisor in $\wi X$ that has only normal crossings and $E=\pi^{-1}(Y)$. Since $\sigma:\wi X\setminus E_0\to \widehat X\setminus\widehat Y$, $\widehat\sigma:\widehat X\setminus\widehat E\to X_\reg$ are biholomorphisms and $X_\sing\subset Y$, $\widehat Y\subset\widehat\sigma^{-1}(Y)$, we conclude that $\pi:\wi X\setminus E\to X\setminus Y$ is a biholomorphism, so property $(i)$ of Proposition \ref{P:divisorization} is satisfied. 

Properties $(ii)$-$(iii)$ of Proposition \ref{P:divisorization} follow easily from $(b)$ and $(c)$, since $\pi(\wi\Sigma_j)=\widehat\sigma(\widehat\Sigma_j)=\Sigma_j$. Moreover, since $\widehat\Sigma_j$ is the strict transform of $\Sigma_j\not\subset X_\sing$ by $\widehat\sigma$ we have $\ord(S,\Sigma_j)=\ord(\widehat\sigma^\star S,\widehat\Sigma_j)$, for any $S\in H^0(X,F)$ and $j=1,\ldots,\ell$.
\end{proof}

Proposition \ref{P:divisorization} has the following corollary which will be needed later, for the proof of Theorem \ref{T:main1}:

\begin{Corollary}\label{C:iso} 
Let $X,L,\Sigma,\tau$ verify assumptions (A)-(D) and let 
$(\wi X,\pi,\wi\Sigma)$ be a divisorization of $(X,\Sigma)$. 
Then $H^0_0(X,L^p,\Sigma,\tau)\cong H^0_0(\wi X,\pi^\star L^p,\wi\Sigma,\tau)$ 
for all $p\geq1$.
\end{Corollary}

\begin{proof} Fix $p\geq1$. The map $\pi$ induces a linear map 
$\pi^\star:H^0(X,L^p)\to H^0(\wi X,\pi^\star L^p)$, $S\to\pi^\star S$. 
We can define a linear map $\pi_\star:H^0(\wi X,\pi^\star L^p)\to H^0(X,L^p)$ 
as follows: if $\wi S\in H^0(\wi X,\pi^\star L^p)$, set $\pi_\star\wi S=S$, 
where $S:=(\pi^{-1})^\star(\wi S\mid_{\wi X\setminus E})\in 
H^0(X\setminus Y,L^p\mid_{X\setminus Y})$ extends to a section in 
$H^0(X,L^p)$ since $X$ is normal and $\dim Y\leq n-2$ \cite[p.\ 143]{GR84}. 
Since $\pi:\wi X\setminus E\to X\setminus Y$ is a biholomorphism, it follows that 
$\pi_\star=(\pi^\star)^{-1}$. Proposition \ref{P:divisorization} $(iii)$ implies that 
$\pi^\star(H^0_0(X,L^p,\Sigma,\tau))\subset H^0_0(\wi X,\pi^\star L^p,\wi\Sigma,\tau)$. 
Moreover, if $\wi S\in H^0_0(\wi X,\pi^\star L^p,\wi\Sigma,\tau)$ then 
$\wi S=\pi^\star\pi_\star\wi S$, hence 
$\ord(\wi S,\wi\Sigma_j)=\ord(\pi_\star\wi S,\Sigma_j)$, so 
$\pi_\star\wi S\in H^0_0(X,L^p,\Sigma,\tau)$. 
Thus $\pi^\star(H^0_0(X,L^p,\Sigma,\tau))=
H^0_0(\wi X,\pi^\star L^p,\wi\Sigma,\tau)$.
\end{proof}

\begin{Remark} 
In hypothesis (D), we make the natural assumption that $\tau_j>\tau_k$, 
for every $j,k\in\{1,\ldots,\ell\}$ with $\Sigma_j\subset\Sigma_k$. 
We note that Corollary \ref{C:iso} is in fact valid for {\it every} $\ell$-tuple 
$\tau$ of positive real numbers. Suppose that $\Sigma_\ell\subset\Sigma_k$ and 
$\tau$ is such that $\tau_\ell<\tau_k$. Set 
$\Sigma'=(\Sigma_1,\ldots,\Sigma_{\ell-1})$, $\tau'=(\tau_1,\ldots,\tau_{\ell-1})$, 
and note that if $(\wi X,\pi,\wi\Sigma)$ 
is a divisorization of $(X,\Sigma)$ then $(\wi X,\pi,\wi\Sigma')$ is a divisorization 
of $(X,\Sigma')$, where $\wi\Sigma'=(\wi\Sigma_1,\ldots,\wi\Sigma_{\ell-1})$. 
Clearly, $H^0_0(X,L^p,\Sigma,\tau)=H^0_0(X,L^p,\Sigma',\tau')$. 
By Corollary \ref{C:iso}, we have that 
$H^0_0(\wi X,\pi^\star L^p,\wi\Sigma,\tau)=
H^0_0(\wi X,\pi^\star L^p,\wi\Sigma',\tau')$.
\end{Remark}

\subsection{Proofs of Theorems \ref{T:main0} and \ref{T:main1}}\label{SS:Tbig}

Theorem \ref{T:main0} will follow from Theorem \ref{T:big} below. 
Let $X$ be a compact complex manifold of dimension $n$, $\Sigma_j\subset X$ 
be irreducible complex hypersurfaces, and let $\tau_j>0$, where $1\leq j\leq\ell$. 
Let $L$ be a holomorphic line bundle over $X$ and consider the following:
\begin{equation}
\begin{split}
&E_p:=L^p\otimes\bigotimes_{j=1}^\ell\cO_X(-\lfloor\tau_jp\rfloor\Sigma_j)\,,\,\;
V_p:=H^0(X,E_p)\,,   \\
&F_p:=L^p\otimes\bigotimes_{j=1}^\ell\cO_X(-t_{j,p}\Sigma_j)\,,\,\;W_p:=H^0(X,F_p)\,,
\end{split}
\end{equation}
where $t_{j,p}$ are defined in \eqref{e:tjp}. Note that $V_p$ is isomorphic 
to the space of sections in  $H^0(X,L^p)$ that vanish to order 
$\lfloor\tau_jp\rfloor$ along $\Sigma_j$, $1\leq j\leq\ell$, while $W_p$ 
is isomorphic to the space $H^0_0(X, L^p)$ defined in \eqref{e:H00}. 
Clearly, $\dim W_p\leq\dim V_p$.

\begin{Theorem}\label{T:big}
In the above setting, the following are equivalent:

(i) $\displaystyle\limsup_{p\to\infty}p^{-n}\dim V_p>0$;

(ii) $\displaystyle\limsup_{p\to\infty}p^{-n}\dim W_p>0$;

(iii) There exists a singular Hermitian metric $h$ on $L$ such that 
$c_1(L,h)-\sum_{j=1}^\ell\tau_j[\Sigma_j]$ is a K\"ahler current on $X$;

(iv) There exist $p_0\in\N$ and $c>0$ such that $\dim V_p\geq cp^n$ for all $p\geq p_0$;

(v) There exist $p_0\in\N$ and $c>0$ such that $\dim W_p\geq cp^n$ for all $p\geq p_0$.
\end{Theorem}

\begin{proof} Let $\omega$ be a Hermitian form on $X$. 

$(i)\Rightarrow(ii)$ There exist a constant $c>0$ and a sequence of natural numbers 
$p_k\nearrow\infty$ such that $\dim V_{p_k}\geq cp_k^n$ for all $k\geq1$. 
Let us fix $k$ and assume that $\tau_1p_k\not\in\N$. Consider the short exact sequence 
\[0\longrightarrow E_{p_k}\otimes\cO_X(-\Sigma_1)\longrightarrow 
E_{p_k}\longrightarrow E_{p_k}\mid_{\Sigma_1}\longrightarrow0\,,\]
which gives the exact sequence
\[0\longrightarrow H^0(X,E_{p_k}\otimes\cO_X(-\Sigma_1))\longrightarrow 
H^0(X,E_{p_k})\longrightarrow H^0(\Sigma_1,E_{p_k}\mid_{\Sigma_1})
\longrightarrow\ldots\,.\]
It follows that 
\[\dim H^0(X,E_{p_k})\leq\dim H^0(X,E_{p_k}\otimes\cO_X(-\Sigma_1))+
\dim H^0(\Sigma_1,E_{p_k}\mid_{\Sigma_1})\,.\]
By Siegel's lemma applied to the analytic subset $\Sigma_1$ 
(see Lemma \ref{L:Siegel} following this proof), there exists a constant $c'>0$ such that 
\[\dim H^0(\Sigma_1,E_p\mid_{\Sigma_1})\leq
\dim H^0(\Sigma_1,L^p\mid_{\Sigma_1})\leq 
c'p^{n-1}\,,\,\;\forall\,p\geq1\,.\]
Hence $\dim H^0(X,E_{p_k}\otimes\cO_X(-\Sigma_1))\geq 
cp_k^n-c'p_k^{n-1}$. If $\tau_2p_k\not\in\N$ we repeat the above argument 
working with the short exact sequence 
\[0\longrightarrow E_{p_k}\otimes\cO_X(-\Sigma_1)\otimes\cO_X(-\Sigma_2)
\longrightarrow E_{p_k}\otimes\cO_X(-\Sigma_1)\longrightarrow
E_{p_k}\otimes\cO_X(-\Sigma_1)\mid_{\Sigma_2}\longrightarrow0\,,\]
and so on. This yields $(ii)$. 

$(ii)\Rightarrow(iii)$ We proceed in two steps.

{\it Step 1.} Assume that $X$ is projective. We fix a smooth ample divisor 
$A\subset X$ and consider, as above, the exact sequence
\[0\longrightarrow H^0(X,F_p\otimes\cO_X(-A))\longrightarrow 
W_p=H^0(X,F_p)\longrightarrow H^0(A,F_p\mid_A)\longrightarrow\ldots\,.\]
Since by Siegel's lemma \cite[Lemma 2.2.6] {MM07}, 
\[\dim H^0(A,F_p\mid_A)\leq\dim H^0(A,L^p\mid_A)=O(p^{n-1})\,,\]
it follows by $(ii)$ that there exist $c>0$ and $p_k\nearrow\infty$ such that 
\[\dim H^0(X,F_{p_k}\otimes\cO_X(-A))\geq cp_k^n\,.\]
We fix such a $p=p_k$. Since $H^0(X,F_p\otimes\cO_X(-A))$ is nontrivial, 
there exists an effective divisor $D_p\subset X$ such that $F_p\otimes\cO_X(-A)=\cO_X(D_p)$. 
Hence 
\[L^p=\cO_X(A)\otimes\cO_X(D_p)\otimes\bigotimes_{j=1}^\ell\cO_X(t_{j,p}\Sigma_j)\,.\]
Let $h_A$ be a smooth positive metric on $\cO_X(A)$, and let $h_{D_p}$, resp.\ $h_{\Sigma_j}$, 
be the metric induced on $\cO_X(D_p)$, resp.\ on $\cO_X(\Sigma_j)$, 
by the canonical section of $\cO_X(D_p)$, resp.\ of $\cO_X(\Sigma_j)$, so that
\[c_1(\cO_X(D_p),h_{D_p})=[D_p]\,,\,\;c_1(\cO_X(\Sigma_j),h_{\Sigma_j})=[\Sigma_j]\,.\]
We define the metric 
$h_p:=h_A\otimes h_{D_p}\otimes\bigotimes_{j=1}^\ell h_{\Sigma_j}^{\otimes t_{j,p}}$ 
on $L^p$, and we let $h=h_p^{1/p}$ be the induced metric on $L$. Then 
\[c_1(L,h)=\frac{1}{p}\,\big(\omega_0+[D_p]+\sum_{j=1}^\ell t_{j,p}[\Sigma_j]\big)\,,\]
where $\omega_0=c_1(\cO_X(A),h_A)$ is a K\"ahler form on $X$. Since $t_{j,p}\geq\tau_jp$ we get 
\[c_1(L,h)-\sum_{j=1}^\ell\tau_j[\Sigma_j]=
\frac{1}{p}\,(\omega_0+[D_p])+\sum_{j=1}^\ell\left(\frac{t_{j,p}}{p}-
\tau_j\right)[\Sigma_j]\geq\frac{1}{p}\,\omega_0\,,\]
which proves $(iii)$ in the case when $X$ is projective. 

{\it Step 2.} In the general case when $X$ is a compact complex manifold 
we have by $(ii)$ that $\limsup_{p\to\infty}p^{-n}\dim H^0(X,L^p)>0$, 
hence $L$ is a big line bundle and $X$ is Moishezon 
(see e.{\ke}g.\ \cite[Theorem 2.2.15]{MM07}). 
By a theorem of Moishezon (see e.{\ke}g.\ \cite[Theorem 2.2.16]{MM07}), 
there exists a projective manifold $\wi X$ and a surjective holomorphic map 
$\sigma:\wi{X}\to{X}$, given as the composition of finitely many blow-ups with 
smooth center, such that $\sigma:\wi X\setminus E\to X\setminus Y$ 
is a biholomorphism, where $Y\subset X$ is an analytic subset with $\dim Y\leq n-2$, 
and $E=\sigma^{-1}(Y)$ is the final exceptional divisor. Let $\Sigma_j'$ 
be the strict transform of $\Sigma_j$ by $\sigma$, and note that if $S\in H^0(X,L^p)$ 
then $\sigma^\star S\in H^0(\wi X,\sigma^\star L^p)$ and 
$\ord(S,\Sigma_j)=\ord(\sigma^\star S,\Sigma_j')$, $1\leq j\leq\ell$. 
It follows that 
\[W_p\cong W_p':=H^0\Big(\wi X,\sigma^\star L^p\otimes
\bigotimes_{j=1}^\ell\cO_{\wi X}(-t_{j,p}\Sigma_j')\Big),\,\;\text{ so }\,
\limsup_{p\to\infty}\frac{\dim W_p'}{p^n}>0\,.\]  

We fix a Hermitian form $\wi\omega$ on $\wi X$. 
Then $\sigma^\star\omega+\wi\omega$ is a Hermitian form on $\wi X$, 
and by Step 1 there exists a singular Hermitian metric $h^\star$ on $\sigma^\star L$ such that 
\begin{equation}\label{e:big1}
c_1(\sigma^\star L,h^\star)-\sum_{j=1}^\ell\tau_j[\Sigma_j']\geq
\varepsilon(\sigma^\star\omega+\wi\omega)\geq\varepsilon\sigma^\star\omega\,,
\end{equation}
for some constant $\varepsilon>0$. The metric 
$h=(\sigma^{-1})^\star h^\star$ on $L\mid_{X\setminus Y}$ 
extends to a metric on $L$ as follows. If $U$ is a coordinate ball centered 
at $x\in Y$ and $e_U$ is a frame of $L\mid_U$, then $\sigma^\star e_U$ 
is a frame of $\sigma^\star L\mid_{\sigma^{-1}(U)}$. 
Let $|\sigma^\star e_U|_{h^\star}=e^{-\varphi^\star}$, 
where $\varphi^\star\in \PSH(\sigma^{-1}(U))$. 
The function $\varphi=\varphi^\star\circ\sigma ^{-1}$ is psh on $U\setminus Y$, 
so it extends to a psh function on $U$ since $\dim Y\leq n-2$, 
and we set $|e_U|_h=e^{-\varphi}$. 

We have that $\sigma_\star c_1(\sigma^\star L,h^\star)=c_1(L,h)$ on $X\setminus Y$, 
and hence on $X$ by the support theorem. Similarly $\sigma_\star[\Sigma_j']=[\Sigma_j]$, 
and $\sigma_\star\sigma^\star\omega=\omega$ since $\omega$ is a smooth form. 
By \eqref{e:big1} it follows that 
$c_1(L,h)-\sum_{j=1}^\ell\tau_j[\Sigma_j]\geq\varepsilon\omega$, 
which proves $(iii)$. 

$(iii)\Rightarrow(iv)$ Let $h$ be a singular metric on $L$ such that 
$R:=c_1(L,h)-\sum_{j=1}^\ell\tau_j[\Sigma_j]\geq3\varepsilon\omega$, 
for some constant $\varepsilon>0$. We fix a smooth metric $h_0$ on $L$, 
set $\alpha=c_1(L,h_0)$, and write $h=h_0e^{-2\psi}$, 
where $\psi\in\PSH(X,\alpha)$ since $c_1(L,h)=\alpha+dd^c\psi\geq0$. 
Let $g_j$ be a smooth metric on $\cO_X(\Sigma_j)$, $s_{\Sigma_j}$ 
be the canonical section of $\cO_X(\Sigma_j)$, and set 
\[\sigma_j:=|s_{\Sigma_j}|_{g_j}\,,\,\;\beta_j=
c_1(\cO_X(\Sigma_j),g_j)\,,\,\;\theta=\alpha-\sum_{j=1}^\ell\tau_j\beta_j.\]
By the Lelong-Poincar\'e formula, $[\Sigma_j]=\beta_j+dd^c\log\sigma_j$. 
Hence $R=\theta+dd^c\psi'$, where $\psi'=\psi-\sum_{j=1}^\ell\tau_j\log\sigma_j$. 
The function $\psi'\in L^1(X,\omega^n)$ is defined everywhere on 
$X\setminus\big(\cup_{j=1}^\ell\Sigma_j\big)$. 
Since $R\geq0$ it follows that $\psi'=u$ a.{\ke}e.\ on $X$, 
for some function $u\in\PSH(X,\theta)$, hence everywhere on 
$X\setminus\big(\cup_{j=1}^\ell\Sigma_j\big)$ since both $\psi',u$ are qpsh there. 
Thus $\psi'$ extends to a $\theta$-psh function on $X$. 

Applying Demailly's regularization theorem \cite{D92} to $\psi'$, 
it follows that there exists a qpsh function $\varphi$ with algebraic singularities on $X$ 
(i.{\ke}e. $\varphi=c\log(\sum_{k=1}^m|f_j|^2)+\chi$ locally on $X$, 
where $c\in\Q$, $c>0$, $f_k$ are holomorphic functions and $\chi$ is a smooth function) 
such that $T:=\theta+dd^c\varphi\geq2\varepsilon\omega$. Note that $\{\varphi=-\infty\}$ 
is an analytic set. 

We take sequences $\{r_{j,k}\}_{k\geq1}\subset\Q$ such that 
$r_{j,k}\searrow\tau_j$ as $k\to\infty$, and we consider the qpsh functions
\[\psi_k=\varphi+\sum_{j=1}^\ell r_{j,k}\log\sigma_j\,,\] 
with algebraic singularities in $Z:=\{\varphi=-\infty\}\cup\bigcup_{j=1}^\ell\Sigma_j$. 
Then
\[\begin{split}
T&= \alpha+dd^c\big(\varphi+
\sum_{j=1}^\ell r_{j,k}\log\sigma_j\big)-\sum_{j=1}^\ell r_{j,k}(\beta_j+
dd^c\log\sigma_j)+\sum_{j=1}^\ell(r_{j,k}-\tau_j)\beta_j \\
 &=\alpha+dd^c\psi_k-\sum_{j=1}^\ell r_{j,k}[\Sigma_j]+
 \sum_{j=1}^\ell(r_{j,k}-\tau_j)\beta_j \,.
\end{split}\]
There exists a constant $C>0$ such that $\beta_j\leq C\omega$ 
for $j=1,\ldots,\ell$. We obtain that 
\[\alpha+dd^c\psi_k= T+\sum_{j=1}^\ell r_{j,k}[\Sigma_j]-
\sum_{j=1}^\ell(r_{j,k}-\tau_j)\beta_j\geq 2\varepsilon\omega-
C\Big(\sum_{j=1}^\ell(r_{j,k}-\tau_j)\Big)\omega\geq\varepsilon\omega\,,\]
if $k$ is chosen sufficiently large.

With $k$ fixed as above, we now define the singular metric 
$h_k=h_0e^{-2\psi_k}$ on $L$, which has algebraic singularities in $Z$. 
Moreover, $c_1(L,h_k)=\alpha+dd^c\psi_k\geq\varepsilon\omega$. 
Let $H^0_{(2)}(X,L^p,h_k^p,\omega^n)$ be the space 
of $L^2$ holomorphic sections of $L^p$ with respect to the metric 
$h_k^p$ on $L^p$ and the volume form $\omega^n$ on $X$. 
Bonavero's singular holomorphic Morse inequalities \cite{Bon93} 
(see also \cite[Theorem 2.3.18]{MM07}) imply that 
\[\dim H^0_{(2)}(X,L^p,h_k^p,\omega^n)\geq\frac{p^n}{n!}\,
\int_{X\setminus Z}c_1(L,h_k)^n+o(p^n)\geq \frac{\varepsilon^np^n}{n!}\,\int_X\omega^n
+o(p^n)\,.\]
Since $\lfloor r_{j,k}p\rfloor\geq\lfloor\tau_j p\rfloor$, it follows from 
the definition of $\psi_k$ that $H^0_{(2)}(X,L^p,h_k^p,\omega^n)\subset V_p$, 
hence $(iv)$ holds. 

$(iv)\Rightarrow(v)$ This follows by the same argument as the one 
in the proof of $(i)\Rightarrow(ii)$. 

$(v)\Rightarrow(i)$ This is obvious since $\dim W_p\subset\dim V_p$.
\end{proof}

\begin{Lemma}\label{L:Siegel}
Let $A$ be a compact irreducible analytic subset of a complex manifold $M$ 
and let $k=\dim A$. If $F$ is a holomorphic line bundle over $A$ then there exists 
a constant $C>0$ depending on $A,M,F$, such that $\dim H^0(A,F^p)\leq Cp^k$ 
for all $p\geq1$. 
\end{Lemma}

\begin{proof} By Theorem \ref{T:H2}, there exists a complex manifold 
$\widetilde M$ and a surjective holomorphic map $\sigma:\wi M\to M$, 
given as the composition of finitely many blow-ups with smooth center, 
such that $E=\sigma^{-1}(A_\sing)$ is a divisor in $\wi M$, 
$\sigma:\wi M\setminus E\to M\setminus A_\sing$ is a biholomorphism, 
and the strict transform $A'=\overline{\sigma^{-1}(A_\reg)}$ 
is a connected $k$-dimensional complex submanifold of $\wi M$. 
The restriction $\sigma:A'\to A$ is a surjective holomorphic map and 
the induced map $\sigma^\star:H^0(A,F^p)\to H^0(A',\sigma^\star F^p)$ 
is injective. So $\dim H^0(A,F^p)\leq\dim H^0(A',\sigma^\star F^p)$ 
and the lemma follows from Siegel's lemma \cite[Lemma 2.2.6] {MM07} 
applied to $A'$ and $\sigma^\star F$.
\end{proof}

\begin{proof}[Proof of Theorem \ref{T:main1}]  
$(i)\Rightarrow(ii)$ By Corollary \ref{C:iso} we have that 
\[H^0_0(X,L^p,\Sigma,\tau)\cong 
H^0_0(\wi X,\pi^\star L^p,\wi\Sigma,\tau)\,,\,\;\forall\,p\geq1\,,\] 
hence $(\pi^\star L,\wi\Sigma,\tau)$ is big and $(ii)$ follows 
from Theorem \ref{T:main0}.

$(ii)\Rightarrow(iii)$ This is obvious.

$(iii)\Rightarrow(iv)$ By Theorem \ref{T:main0} 
to $\wi X,\pi^\star L,\wi\Sigma,\tau$, there exist $p_0\in\N$ and 
$c>0$ such that $\dim H^0_0(\wi X,\pi^\star L^p,\wi\Sigma,\tau)\geq cp^n$ 
for all $p\geq p_0$. Hence $(iv)$ follows using Corollary \ref{C:iso}.

$(iv)\Rightarrow(i)$ This is obvious by Definition \ref{D:big}.
\end{proof}

Theorem \ref{T:main1} has the following interesting corollary:

\begin{Corollary}\label{C:big}
Let $X,L,\Sigma,\tau$ verify assumptions (A)-(D), and suppose that 
$\dim\Sigma_j=n-1$ for all $j=1,\ldots,\ell$. Let $\Sigma_j'\subset\Sigma_j$ 
be distinct irreducible proper analytic subsets such that 
$\Sigma_j'\not\subset X_\sing$, and let $\Sigma'=
(\Sigma_1,\ldots,\Sigma_\ell,\Sigma_1',\ldots,\Sigma_\ell')$, $\tau'=
(\tau_1,\ldots,\tau_\ell,\tau_1+\delta,\ldots,\tau_\ell+\delta)$, 
where $\delta>0$. If $(L,\Sigma,\tau)$ is big then $(L,\Sigma',\tau')$ 
is big for $\delta>0$ sufficiently small.
\end{Corollary}

\begin{proof} Without loss of generality we may assume that $X$ 
is a complex manifold, by first desingularizing $X$ if necessary using 
Theorem \ref{T:H1} and applying Corollary \ref{C:iso} to the map $\sigma$ from 
Theorem \ref{T:H1} and the strict transforms of $\Sigma_j,\Sigma_j'$ by $\sigma$. 
Let $\omega$ be a Hermitian form on $X$. 

Let $(\wi X,\pi,\wi\Sigma')$ be a divisorization of $(X,\Sigma')$, 
so $\pi:\wi X\setminus E\to X\setminus Y$ is a biholomorphism, 
where $E=\pi^{-1}(Y)$ is the final exceptional divisor, $\wi\Sigma_j'$ 
are irreducible components of $E$, and $\wi\Sigma_j$ is the strict transform 
of $\Sigma_j$ by $\pi$.  Let $\wi\omega$ be a Hermitian form on 
$\wi X$ such that $\wi\omega\geq\pi^\star\omega$. 
By Theorem \ref{T:main1}, there exists a singular Hermitian metric 
$h^\star$ on $\pi^\star L$ such that 
\[T:=c_1(\pi^\star L,h^\star)-
\sum_{j=1}^\ell\tau_j[\wi\Sigma_j]\geq\varepsilon\wi\omega\geq
\varepsilon\pi^\star\omega\,,\]
for some constant $\varepsilon>0$. 
As argued in the proof of Theorem \ref{T:big} (Step 2 of the implication 
$(ii)\Rightarrow(iii)$), the metric $h=(\pi^{-1})^\star h^\star$ on $L\mid_{X\setminus Y}$ 
extends to a metric on $L$. Moreover, 
$\pi_\star c_1(\pi^\star L,h^\star)=
c_1(L,h)$, $\pi_\star[\wi\Sigma_j]=[\Sigma_j]$, and 
$\pi_\star\pi^\star\omega=\omega$. We conclude that 
\[S:=\pi_\star T=c_1(L,h)-\sum_{j=1}^\ell\tau_j[\Sigma_j]\geq\varepsilon\omega\,.\]

It is well known that there exists a smooth Hermitian metric $g$ 
on the line bundle $\cO_{\wi X}(E)$ and a constant $\delta_0>0$ 
such that $\wi\omega_0:=\pi^\star\omega-\delta_0\Theta$ 
is a Hermitian form on $\wi X$, where $\Theta=c_1(\cO_{\wi X}(E),g)$ 
(see e.{\ke}g.\ \cite[Lemma 2.2]{CMM17}). If $s_E$ is the canonical section 
of $\cO_{\wi X}(E)$ then by the Lelong-Poincar\'e formula, $[E]=\Theta+dd^c\log|s_E|_g$. 

Let $\delta=\varepsilon\delta_0$. Then $\pi^\star S-
\delta\Theta\geq\varepsilon\pi^\star\omega-\varepsilon\delta_0\Theta=
\varepsilon\wi\omega_0$. We introduce the singular Hermitian metric 
$\wi h=|s_E|_g^{-2\delta}\pi^\star h$ on $\pi^\star L$, so 
$c_1(\pi^\star L,\wi h)=c_1(\pi^\star L,\pi^\star h)+\delta dd^c\log|s_E|_g$. 
Since $\Sigma_j'\subset\Sigma_j$ we have $\pi^\star [\Sigma_j]=
[\wi\Sigma_j]+[\wi\Sigma_j']+R_j$, where $R_j$ 
are positive closed currents of bidegree $(1,1)$ supported in $E$. 
It follows that 
\begin{align*}
\pi^\star S-\delta\Theta&=c_1(\pi^\star L,\pi^\star h)-
\sum_{j=1}^\ell\tau_j\pi^\star [\Sigma_j]-\delta\big([E]-dd^c\log|s_E|_g\big) \\
&=c_1(\pi^\star L,\wi h)-\sum_{j=1}^\ell\tau_j[\wi \Sigma_j]-
\sum_{j=1}^\ell(\tau_j+\delta)[\wi \Sigma_j']-R\,,
\end{align*}
where $R$ is a positive closed current of bidegree $(1,1)$ supported in $E$. Thus 
\[c_1(\pi^\star L,\wi h)-\sum_{j=1}^\ell\tau_j[\wi \Sigma_j]-
\sum_{j=1}^\ell(\tau_j+\delta)[\wi \Sigma_j']=
\pi^\star S-\delta\Theta+R\geq\varepsilon\wi\omega_0\,,\]
hence $(L,\Sigma',\tau')$ is big by Theorem \ref{T:main1}.
\end{proof}

\section{Envelopes of qpsh functions with poles along a divisor}\label{S:env}
In this Section we define the relevant spaces qpsh functions with poles along a divisor
and prove the regularity theorem for their upper enveloppes.
Let $(X,\omega)$ be a compact Hermitian manifold of dimension 
$n$, $\Sigma_j\subset X$ be irreducible complex hypersurfaces, and 
let $\tau_j>0$, where $1\leq j\leq\ell$. We write 
$\Sigma=(\Sigma_1,\ldots,\Sigma_\ell)$, $\tau=(\tau_1,\ldots,\tau_\ell)$, 
and we denote by $\dist$ the distance on $X$ induced by $\omega$.

Let $\alpha$ be a smooth closed real $(1,1)$-form on $X$. 
We fix a smooth Hermitian metric $g_j$ on $\cO_X(\Sigma_j)$, 
let $s_{\Sigma_j}$ be the canonical section of $\cO_X(\Sigma_j)$, $1\leq j\leq\ell$, and set 
\begin{equation}\label{e:bts}
\beta_j=c_1(\cO_X(\Sigma_j),g_j)\,,\,\;\theta=
\alpha-\sum_{j=1}^\ell\tau_j\beta_j\,,\,\;\sigma_j:=
|s_{\Sigma_j}|_{g_j}\,.
\end{equation}
We let 
\begin{equation}\label{e:calL}
\mathcal{L}(X,\alpha,\Sigma,\tau)=\{\psi\in\PSH(X,\alpha):\,
\nu(\psi,x)\geq\tau_j,\,\forall\,x\in\Sigma_j,\,1\leq j\leq\ell\}\,
\end{equation}
be the class of $\alpha$-psh functions with logarithmic poles of order $\tau_j$ 
along $\Sigma_j$. Given a function $\varphi:X\to\R\cup\{-\infty\}$ we consider 
the following subclasses of qpsh functions and their upper envelopes: 
\begin{align}
\mathcal{A}(X,\alpha,\Sigma,\tau,\varphi)&=
\{\psi\in\mathcal{L}(X,\alpha,\Sigma,\tau):\,\psi\leq\varphi\text{ on }X\}\,, \label{e:calA} \\
\mathcal{A}'(X,\alpha,\Sigma,\tau,\varphi)&=
\Big\{\psi'\in\PSH(X,\theta):\,\psi'\leq\varphi-
\sum_{j=1}^\ell\tau_j\log\sigma_j\text{ on }
X\setminus\bigcup_{j=1}^\ell\Sigma_j\Big\}\,,  \label{e:calA'} \\
\varphi_\eq=\varphi_{\eq,\Sigma,\tau}&=
\sup\{\psi:\,\psi\in\mathcal{A}(X,\alpha,\Sigma,\tau,\varphi)\}\,,  \label{e:enveq} \\
\varphi_\req=\varphi_{\req,\Sigma,\tau}&=
\sup\{\psi':\,\psi'\in\mathcal{A}'(X,\alpha,\Sigma,\tau,\varphi)\}\,.  \label{e:envreq}
\end{align}

The function $\varphi_\eq$ defined in \eqref{e:enveq} is the largest 
$\alpha$-psh function dominated by $\varphi$ and with logarithmic poles of order 
$\tau_j$ along $\Sigma_j$. We call $\varphi_\eq$ the {\em equilibrium envelope of 
$(\alpha,\Sigma,\tau,\varphi)$}, and $\varphi_\req$ the 
{\em reduced equilibrium envelope of $(\alpha,\Sigma,\tau,\varphi)$}. 
This is motivated by the terminology of {\em equilibrium metric} 
used in the case when $\varphi$ is the weight of a singular metric 
$h=h_0e^{-2\varphi}$ on a Hermitian holomorphic line bundle $(F,h_0)$ over $X$ (see below).

Extremal psh functions on domains in Stein manifolds with poles along subvarieties 
(also known as pluricomplex Green functions) are studied in \cite{LS99}, \cite{RS05}. 
In particular, pluricomplex Green functions with finitely many poles were studied 
by many authors. In the context of metrics on line bundles over compact complex manifolds, 
the above envelope method is introduced in \cite[Section 4.1]{Ber07} 
for defining equilibrium metrics with poles along a divisor. 
More generally, equilibrium metrics with prescribed singularities 
on a line bundle are introduced and studied in \cite{RWN17} (see also \cite[Theorem 3]{Dar17}). 

Our first result is concerned with some basic properties of the envelope 
defined in \eqref{e:enveq} under natural, very general assumptions.

\begin{Proposition}\label{P:envelopes}
Let $X,\Sigma,\tau,\alpha,\theta$ be as above, 
and let $\varphi:X\to\R\cup\{-\infty\}$ be an upper semicontinuous function. 
Then the following hold: 

(i) The mapping 
$\PSH(X,\theta)\ni\psi'\mapsto\psi:=\psi'+
\sum_{j=1}^\ell\tau_j\log\sigma_j\in\mathcal{L}(X,\alpha,\Sigma,\tau)$ 
is well defined and bijective, with inverse 
$\mathcal{L}(X,\alpha,\Sigma,\tau)\ni\psi\mapsto\psi':=
\psi-\sum_{j=1}^\ell\tau_j\log\sigma_j\in\PSH(X,\theta)$.

(ii) There exists a constant $C>0$ depending only on 
$X,\Sigma,\tau,\alpha,\theta$ such that $\sup_X\psi'\leq\sup_X\varphi+C$, 
for every $\psi'\in\mathcal{A}'(X,\alpha,\Sigma,\tau,\varphi)$. 

(iii) $\mathcal{A}(X,\alpha,\Sigma,\tau,\varphi)\neq\emptyset$ 
if and only if $\mathcal{A}'(X,\alpha,\Sigma,\tau,\varphi)\neq\emptyset$. 
Moreover, in this case we have that 
$\varphi_\req\in\mathcal{A}'(X,\alpha,\Sigma,\tau,\varphi)$, 
$\varphi_\eq\in\mathcal{A}(X,\alpha,\Sigma,\tau,\varphi)$, and
\begin{equation}\label{e:eqreq}
\varphi_\eq=\varphi_\req+\sum_{j=1}^\ell\tau_j\log\sigma_j\,\text{ on $X$.}
\end{equation}

(iv) If $\varphi$ is bounded and there exists a bounded 
$\theta$-psh function, then $\varphi_\req$ is bounded on $X$. 

(v) If $\PSH(X,\theta)\neq\emptyset$ and $\varphi_1,\varphi_2:X\to\R$ 
are bounded and upper semicontinuous, then 
\[\varphi_{1,\req}-\sup_X|\varphi_1-\varphi_2|\leq\varphi_{2,\req}\leq
\varphi_{1,\req}+\sup_X|\varphi_1-\varphi_2|\]
holds on $X$. Moreover, if $\varphi_1\leq\varphi_2$ then 
$\varphi_{1,\req}\leq\varphi_{2,\req}$.
\end{Proposition} 

\begin{proof} $(i)$ If $\psi'\in\PSH(X,\theta)$ then $\psi=\psi'+
\sum_{j=1}^\ell\tau_j\log\sigma_j$ is qpsh and 
\[\alpha+dd^c\psi=\theta+dd^c\psi'+\sum_{j=1}^\ell\tau_j[\Sigma_j]\geq0\,,\]
so $\psi\in\PSH(X,\alpha)$. If $x\in\Sigma_j$ then $\nu(\psi,x)\geq
\tau_j\nu(\log\sigma_j,x)\geq\tau_j$, hence $\psi\in\mathcal{L}(X,\alpha,\Sigma,\tau)$. 

Conversely, if $\psi\in\mathcal{L}(X,\alpha,\Sigma,\tau)$, let 
$T=\alpha+dd^c\psi$. Since $\nu(T,x)\geq\tau_j$ for all $x\in\Sigma_j$ 
we have by Siu's decomposition theorem that $T'=T-\sum_{j=1}^\ell\tau_j[\Sigma_j]$ 
is a positive closed current. Moreover, $T'=\theta+dd^c\psi'$, 
where $\psi'=\psi-\sum_{j=1}^\ell\tau_j\log\sigma_j$. 
The function $\psi'\in L^1(X,\omega^n)$ is defined everywhere on 
$X\setminus\bigcup_{j=1}^\ell\Sigma_j$. Since $T'\geq0$ it follows that 
$\psi'=u$ a.{\ke}e.\ on $X$, for some function $u\in\PSH(X,\theta)$, 
hence everywhere on $X\setminus\bigcup_{j=1}^\ell\Sigma_j$ 
since both $\psi',u$ are qpsh there. Thus $\psi'$ extends to a $\theta$-psh function on $X$. 

$(ii)$ There exist points $x_k\in X$, coordinate neighborhoods $U_k$ 
centered at $x_k$, and numbers $r_k>0$, $1\leq k\leq N$, 
such that the balls $\overline\B(x_k,2r_k)\subset U_k$ and 
$X=\bigcup_{k=1}^N\B(x_k,r_k)$. Set $r=\min_{1\leq k\leq N}r_k$. 
Let $\rho_k$ be a smooth function defined in a neighborhood of 
$\overline\B(x_k,2r_k)$  such that $dd^c\rho_k=\theta$. 
If $\psi'\in \PSH(X,\theta)$ and $x\in\B(x_k,r_k)$ we have by 
the subaverage inequality for psh functions that 
\[\rho_k(x)+\psi'(x)\leq\frac{n!}{\pi^nr^{2n}}
\int_{\B(x,r)}(\rho_k+\psi')\,d\lambda\,,\]
where $\lambda$ is the Lebesgue measure in coordinates. 
Hence there exists a constant $C'>0$ such that for every function 
$\psi'\in\PSH(X,\theta)$ one has 
\[\psi'(x)\leq\frac{n!}{\pi^nr^{2n}}\int_{\B(x,r)}\psi'\,d\lambda+
C'\,,\,\;\forall\,x\in\B(x_k,r)\,,\,\;k=1,\ldots,N.\]
If $\psi'\in\mathcal{A}'(X,\alpha,\Sigma,\tau,\varphi)$ and $x\in\B(x_k,r)$, we have 
\begin{align*}
\psi'(x)&\leq\frac{n!}{\pi^nr^{2n}}\int_{\B(x,r)}
\Big(\varphi-\sum_{j=1}^\ell\tau_j\log\sigma_j\Big)\,d\lambda+C'\\
&\leq\sup_{\B(x,r)}\varphi+\frac{n!}{\pi^nr^{2n}}
\int_{\B(x,r)}\Big|\sum_{j=1}^\ell\tau_j\log\sigma_j\Big|\,d\lambda+
C' \leq\sup_{\B(x,r)}\varphi+C\,,
\end{align*}
for some constant $C>0$ depending only on $X,\Sigma,\tau,\alpha,\theta$. 
Hence $\sup_X\psi'\leq\sup_X\varphi+C$.

$(iii)$ It follows immediately from $(i)$ that the mapping 
\[\mathcal{A}'(X,\alpha,\Sigma,\tau,\varphi)\ni\psi'\longmapsto\psi:=
\psi'+\sum_{j=1}^\ell\tau_j\log\sigma_j\in\mathcal{A}(X,\alpha,\Sigma,\tau,\varphi),\]
is well defined and bijective. By $(ii)$, the family 
$\mathcal{A}'(X,\alpha,\Sigma,\tau,\varphi)$ of $\theta$-psh functions 
is uniformly upper bounded, hence the upper semicontinuous regularization 
$\varphi_\req^\star$ of $\varphi_\req$ is $\theta$-psh. 
Since $\varphi_\req\leq\varphi-\sum_{j=1}^\ell\tau_j\log\sigma_j$ 
on $X\setminus\bigcup_{j=1}^\ell\Sigma_j$ and the latter 
is upper semicontinuous there, we see that 
$\varphi_\req^\star\in\mathcal{A}'(X,\alpha,\Sigma,\tau,\varphi)$, 
so $\varphi_\req=\varphi_\req^\star$. Moreover, if 
$\psi\in\mathcal{A}(X,\alpha,\Sigma,\tau,\varphi)$ then 
$\psi\leq\varphi_\req+\sum_{j=1}^\ell\tau_j\log\sigma_j$. 
It follows that the family $\mathcal{A}(X,\alpha,\Sigma,\tau,\varphi)$ 
is uniformly upper bounded, the upper semicontinuous regularization 
$\varphi_\eq^\star$ of $\varphi_\eq$ is $\alpha$-psh and it verifies 
$\varphi_\eq^\star\leq\varphi_\req+\sum_{j=1}^\ell\tau_j\log\sigma_j$ 
and $\varphi_\eq^\star\leq\varphi$ on $X$, since the functions 
on the right hand side are upper semicontinuous. 
Hence $\varphi_\eq^\star\in\mathcal{A}(X,\alpha,\Sigma,\tau,\varphi)$, 
so $\varphi_\eq=\varphi_\eq^\star$ and \eqref{e:eqreq} is clearly satisfied.

$(iv)$ Since $m:=\inf_X\big(\varphi-\sum_{j=1}^\ell\tau_j\log\sigma_j\big)>-\infty$, 
there exists a bounded $\theta$-psh function $\psi'$ such that $\psi'\leq m$ on $X$. 
Thus $\psi'\leq\varphi_\req\leq\sup_X\varphi+C$ on $X$. 

$(v)$ Since $\PSH(X,\theta)\neq\emptyset$ and $\varphi_j$ 
is bounded, it follows that 
$\mathcal{A}'(X,\alpha,\Sigma,\tau,\varphi_j)\neq\emptyset$, $j=1,2$. 
Then $(v)$ follows easily from the definition \eqref{e:envreq} of $\varphi_\req$.
\end{proof}

The following notion is needed for studying certain regularity 
properties of the equilibrium envelopes.

\begin{Definition}\label{D:Hosing}
A function $\phi:\ X\to[-\infty,\infty)$ is H\"older with singularities along 
a proper analytic subset $A\subset X$ if there exist constants 
$c,\varrho>0$ and $0<\nu\leq1$ such that  
\begin{equation}\label{e:Hdef}
|\phi(z)-\phi(w)|\leq\frac{c\,\dist(z,w)^\nu}{\min \{\dist(z,A),\dist(w,A)\} ^\varrho}\;,\,\;
\text{for all $z,w\in X\setminus A$}\,.
\end{equation}
\end{Definition} 

We assume now that the class  $\{\theta\}_\ddbar$ is big and we let 
$Z_0:=\NAmp\big(\{\theta\}_\ddbar\big)$. 
By \cite[Theorem\ 3.17]{Bo04} there exists a K\"ahler current 
$T_0\in\{\theta\}_\ddbar$ with analytic singularities such that $E_+(T_0)=Z_0$. We let
\begin{equation}\label{e:Kc}
T_0=\theta+dd^c\psi_0\geq\varepsilon_0\omega\,,
\end{equation}
where $\varepsilon_0>0$ and $\psi_0$ is a qpsh function 
with analytic singularities. So $Z_0=\{\psi_0=-\infty\}$. 
By subtracting a constant we may assume that $\psi_0\leq-1$ on $X$.
 
Using the methods developed in \cite{D94} and \cite{BD12} 
(see also \cite{DMN17}), we prove next the following regularity 
result for the functions $\varphi_\req=\varphi_{\req,\Sigma,\tau}$ 
defined in \eqref{e:envreq}, and $\varphi_\eq=\varphi_{\eq,\Sigma,\tau}$ 
defined in \eqref{e:enveq}.
 
 \begin{Theorem}\label{T:reg}
 Let $(X,\omega)$ be a compact Hermitian manifold of dimension $n$, 
 $\Sigma_j\subset X$ be irreducible complex hypersurfaces, 
 and let $\tau_j>0$, where $1\leq j\leq\ell$.  Let $\alpha$ 
 be a smooth closed real $(1,1)$-form on $X$ and $\theta$ 
 be as in \eqref{e:bts}. Assume that the class $\{\theta\}_\ddbar$ 
 is big and let $Z_0:=\NAmp\big(\{\theta\}_\ddbar\big)$. 
 Then the following hold:
 
 (i) If $\varphi:X\to\R$ is H\"older continuous then $\varphi_\req$ 
 is H\"older with singularities along $Z_0$, and $\varphi_\eq$ is 
 H\"older with singularities along $\Sigma_1\cup\ldots\cup \Sigma_\ell\cup Z_0$.

(ii) If $\varphi:X\to\R$ is continuous then $\varphi_\req$ 
is continuous on $X\setminus Z_0$, and $\varphi_\eq$ is 
continuous on $X\setminus(\Sigma_1\cup\ldots\cup \Sigma_\ell\cup Z_0)$.
\end{Theorem}

We will need the following lemma which follows from \cite[Lemma 2.8]{DMN17}:

\begin{Lemma}\label{L:Hoelder}
Let $r,\nu\in(0,1)$ and $u$ be a subharmonic function in a neighborhood 
of the ball $\overline\B(0,3r)\subset \R^m$. 
Suppose that there exists a constant  $A>0$ such that 
$|u(x)|\leq A$ for all $x\in\B(0,3r)$, and 
\[\Delta u(\B(x,s))\leq A s^{m-2+\nu}\,,\]
for all $x\in \B(0,2r)$ and $0<s<r$. Then there exists a constant $C(m,\nu)>0$ such that
\[|u(x)-u(y)|\leq\frac{C(m,\nu)A}{r^\nu}\,|x-y|^\nu\,,\,\;\forall\,x,y\in \B(0,r)\,.\]
\end{Lemma}

\begin{proof}[Proof of Theorem \ref{T:reg}]
 
$(i)$ Assume $\varphi$ is a $\cC^\nu$ function on $X$ and let 
$\psi:=\phi_\req$. We show that $\psi$ is H\"older with singularities along 
$Z_0$ and with H\"older exponent $\nu$. Following \cite{BD12}, 
we will regularize $\psi$ using the method introduced by Demailly in \cite{D94}. 

Consider the exponential map associated with the Chern connection 
on the tangent bundle $TX$ of $X$. The {\it formal holomorphic part} 
of its Taylor expansion is denoted by
\[\exph: TX\to X \quad \text{with} \quad T_xX\ni\zeta\mapsto\exph_x(\zeta)\,.\]
Let $\chi:\ \R\to [0,\infty)$ be the smooth function with support in $(-\infty,1]$ defined by
\[\chi(t)=\frac{\const}{(1-t)^2}\,\exp\frac{1}{t-1}\,
\text{ for $t<1$}\,,\,\;\chi(t)=0\,\text{ for $t\geq1$}\,,\]
where the constant $\const$  is adjusted so that 
$\int_{|\zeta|\leq 1}\chi(|\zeta|^2)\,d\Leb(\zeta)=1$ with respect to the 
Lebesgue measure $\Leb(\zeta)$ on $\C^n\simeq T_xX$. 
We fix a constant $\delta_0>0$ small enough and define 
\begin{align}\label{eq:n3.38}
\Psi(x,t):=\int_{\zeta\in T_xX}  \psi(\exph_x(t\zeta))\chi(|\zeta|^2)\,d\Leb(\zeta)\,,\,\;\text{ for $(x,t)
\in X\times [0,\delta_0]$}\,.
\end{align}
By \cite{D94}, there exists a constant $b$ such that the function 
$t\mapsto \Psi(x,t)+bt^2$ is increasing for $t$ in $[0,\delta_0]$. 
Observe also that $\Psi(x,0)=\psi(x)$. 

Consider  for $c>0$ and $\delta\in (0,\delta_0]$ the {\it Kiselman-Legendre transform} 
\begin{equation}\label{eq_Kiselman-Legendre_transform}
\psi_{c,\delta}(x):=\inf_{t\in (0,\delta]}\Big( \Psi(x,t)+bt^2-b\delta^2 - c\log\frac{t}{\delta}\Big).
\end{equation}
It was shown in  \cite[Lemma 1.12]{BD12} that $\psi_{c,\delta}$ is qpsh and 
\begin{equation}\label{e:dd^c_psi_c,delta}
\theta+dd^c\psi_{c,\delta}\geq-(ac+b\delta^2)\omega\,, 
\end{equation}
where $a>0$ is  a constant (see also \cite{Kiselman78, Kiselman94}).

For $t:=\delta$ we obtain from \eqref{eq_Kiselman-Legendre_transform} 
that $\psi_{c,\delta}(x)\leq \Psi(x,\delta)$. From \eqref{eq:n3.38} we deduce that $\Psi(x,\delta)$ 
is an average of values of $\psi$ in a ball $\B(x,A\delta)$ in $X$ for some 
constant $A$ depending only on $X$ and $\omega$. 
By Proposition \ref{P:envelopes} we have that 
\[\psi\leq\min\Big\{\varphi-\sum_{j=1}^\ell\tau_j\log\sigma_j\,,\,\sup_X\varphi+
C\Big\}\leq\varphi+f\] 
on $X$, where 
$f:=\min\big\{-\sum_{j=1}^\ell\tau_j\log\sigma_j\,,\,
\sup_X\varphi-\inf_X\varphi+C\big\}$. 
In the sequel we denote $O(\delta^\nu):=C'\delta^\nu$ 
with constants $C'>0$ independent of $x$ and $\delta$. 
Since $\varphi\in\cC^\nu$ and $f$ is a Lipschitz function on $X$ it follows that  
\[\Psi(x,\delta)\leq\sup_{\B(x,A\delta)}(\varphi+f)\leq\varphi(x)+
f(x)+O(\delta^\nu)\leq\varphi(x)-\sum_{j=1}^\ell\tau_j\log\sigma_j(x)+O(\delta^\nu)\,,\]
for all $x\in X$ and $0<\delta\leq\delta_0$. 
Hence 
$\psi_{c,\delta}\leq\varphi-\sum_{j=1}^\ell\tau_j\log\sigma_j+O(\delta^\nu)$. Let $\psi_0$ 
be the $\theta$-psh function with analytic singularities in $Z_0$ defined in \eqref{e:Kc}. 
Since $\varphi$ is bounded there exists a constant $C_1>0$ such that $\psi_0\leq C_1+
\varphi-\sum_{j=1}^\ell\tau_j\log\sigma_j$. So $\psi_0-C_1\leq\psi$, 
and $\psi$ is locally bounded on $X\setminus Z_0$. 
Consider the convex combination
\[\xi:=\frac{ac+b\delta^2}{\varepsilon_0 }\,\psi_0+
\Big( 1- \frac{ac+b\delta^2}{\varepsilon_0 }\Big)\psi_{c,\delta}\,,\]
where we take $c=\delta^\nu$. We deduce from the above upper 
bounds for $\psi_{c,\delta}$ and $\psi_0$ that
\[\xi\leq\varphi-\sum_{j=1}^\ell\tau_j\log\sigma_j+O(\delta^\nu)\,.\]
By \eqref{e:Kc} and \eqref{e:dd^c_psi_c,delta} we have that 
\[\theta+dd^c\xi\geq  (ac+b\delta^2)\omega-
\Big( 1- \frac{ac+b\delta^2}{\varepsilon_0}\Big)(ac+b\delta^2)\omega\geq 0\,,\]
hence 
\[\psi(x)+O(\delta^\nu)\geq\xi(x)\geq
\frac{ac+b\delta^2}{\varepsilon_0 }\,\psi_0(x)+
\psi_{c,\delta}(x)-\frac{ac+b\delta^2}{\varepsilon_0}\,\Psi(x,\delta)\,,\,\;
\forall\,x\in X\,.\]
Since $c=\delta^\nu$ and by Proposition \ref{P:envelopes} 
$(ii)$, $\Psi(x,\delta)\leq\sup_X\psi\leq\sup_X\varphi+C$, it follows that
\[\frac{ac+b\delta^2}{\varepsilon_0 }\,\psi_0+
\psi_{c,\delta}\leq\psi+O(\delta^\nu)\,.\]

If $x\in X\setminus Z_0$ then $\Psi(x,0)=\psi(x)>-\infty$, 
so the increasing function $t\mapsto\Psi(x,t)+bt^2$ is bounded 
and the infimum in the definition of $\psi_{c,\delta}(x)$ is reached for 
some $t=t_{x,\delta}\in(0,\delta]$. Hence
 \begin{equation}\label{eq_rho_psi}
 \Psi(x,t_{x,\delta})+bt_{x,\delta}^2\leq \psi(x) +
 c\log\frac{t_{x,\delta}}{\delta}-\frac{ac+b\delta^2}{\varepsilon_0 }\,\psi_0(x)+O(\delta^\nu)\,.
 \end{equation}
Using the fact that $t\mapsto \Psi(x,t)+bt^2$ is increasing, this implies that 
 \[c\log\frac{t_{x,\delta}}{\delta}-\frac{ac+b\delta^2}{\varepsilon_0 }\,\psi_0(x)
 +O(\delta^\nu)\geq 0\,.\]
Since $c=\delta^\nu$ and $\psi_0\leq-1$, we infer from the above inequality 
that there exists a constant $C_2>0$ such that 
$e^{C_2\psi_0(x)}\delta\leq t_{x,\delta}\leq\delta$, 
for all $x\in X\setminus Z_0$. By \eqref{eq_rho_psi} and using again that 
$t\mapsto \Psi(x,t)+bt^2$ is increasing, we see that 
\[\Psi\big(x,e^{C_2\psi_0(x)}\delta\big)-\psi(x)\leq
|\psi_0(x)|O(\delta^\nu)\,,\,\;\forall x\in X\setminus Z_0\,.\]
Since $\psi_0$ has analytic singularities, it follows by the 
Lojasiewicz inequality \cite[Theorem 5.2.4]{KP02} that there exists a constant $M>0$ such that
$e^{\psi_0(x)}\gtrsim\dist(x,Z_0)^M$. 
Letting $t=e^{C_2\psi_0(x)}\delta$ 
we conclude by the above estimates that there exist constants $\varepsilon_1,N>0$ such that 
\begin{equation}\label{e:Psi_psi}
\Psi(x,t)-\psi(x)\leq\frac{O(t^\nu)}{\dist(x,Z_0)^N}\;,\,\;
\forall\,x\in X\setminus Z_0,\;0<t<\varepsilon_1\dist(x,Z_0)^N\,.
\end{equation}

There exists $r_0>0$ such that every $x\in X$ has a coordinate neighborhood $U_x$ 
centered at $x$ with $\overline\B(x,3r_0)\subset U_x$ and so that the metric on $X$ 
coincides at $x$ with the standard metric given by the coordinates. 
According to \cite[(4.5)]{D94} (see also \cite[(1.16)]{BD12}, \cite[Theorem 2.7]{DMN17}) 
it follows from a Lelong-Jensen type inequality that 
\[\Psi(x,t)-\psi(x)\gtrsim\frac{1}{t^{2n-2}}\,\int_{\B(x,t/4)}\Delta\psi - O(t^2)\,,\] 
where the constants involved are independent of $x\in X$ and $t\in(0,\delta_0)$.
 Combining this and \eqref{e:Psi_psi}, we infer that  
 \[\int_{\B(x,t)}\Delta\psi\leq\frac{O(t^{2n-2+\nu})}{\dist(x,Z_0)^N}\;,\,\;
 \forall\,x\in X\setminus Z_0,\;0<t<\min\big\{r_0,\varepsilon_1\dist(x,Z_0)^N\big\}\,.\]
 
 Lemma \ref{L:Hoelder} implies that there exist constants $C_3,\varepsilon_2>0$ 
 such that if $x\in X\setminus Z_0$ and $\dist(y,x)\leq\varepsilon_2\dist(x,Z_0)^N$ then 
 \[|\psi(y)-\psi(x)|\leq\frac{C_3}{\dist(x,Z_0)^{2N}}\,\dist(y,x)^\nu\,.\]

Since $\psi_0-C_1\leq\psi$, it follows that 
$|\psi|\leq C_4|\log\dist(\LargerCdot,Z_0)|+C_1$ holds on $X$, 
for some constant $C_4>0$. Assume now that 
$x,y\in X\setminus Z_0$ and $\dist(y,x)\geq\varepsilon_2\dist(x,Z_0)^N$. Then 
\[|\psi(y)-\psi(x)|\leq\big(C_4|\log\dist(x,Z_0)|+
C_4|\log\dist(y,Z_0)|+2C_1\big)\frac{\dist(y,x)}{\varepsilon_2\dist(x,Z_0)^N}\;.\]
The previous two estimates combined show that $\psi=\varphi_\req$ 
is H\"older with singularities along $Z_0$ and with H\"older exponent $\nu$. 
Hence by \eqref{e:eqreq}, $\varphi_\eq$ is H\"older with singularities along 
$\Sigma_1\cup\ldots\cup \Sigma_\ell\cup Z_0$ and with H\"older exponent 
$\nu$, since $\log\sigma_j$ is Lipschitz with singularities along $\Sigma_j$. 

\medskip

$(ii)$ Let $\{\varphi_k\}$ be a sequence of real-valued smooth functions 
converging uniformly to $\varphi$ on $X$. By $(i)$ and Proposition 
\ref{P:envelopes} $(v)$, $\varphi_{k,\req}$ are continuous and converge 
uniformly to $\varphi_\req$ on $X\setminus Z_0$, hence $\varphi_\req$ 
is continuous on $X\setminus Z_0$. 
 \end{proof}

\section{Convergence of the global Fubini-Study potentials}\label{S:FSpotentials}
 
Let $X,L,\Sigma,\tau$ verify assumptions (A)-(D), 
and assume in addition that there exists a K\"ahler form $\omega$ 
on $X$ and that $h$ is a continuous Hermitian metric on $L$. 
In this section we prove the convergence of the Fubini-Study potentials and currents 
of the space $H^0_0 (X, L^p)$ defined in \eqref{e:H00}, 
towards $\varphi_\eq$ and the equilibrium current $T_\eq$ of $(L,h,\Sigma,\tau)$,
respectively (Theorem \ref{T:FSpot}).

Let $h_0,\varphi$ be as in \eqref{e:varphi}. Let $P_p,\gamma_p$ 
be the Bergman kernel function and Fubini-Study current of the space 
$H^0_{0,(2)}(X,L^p)$, and let $\varphi_p$ be the global Fubini-Study potential 
of $\gamma_p$ (see \eqref{e:FSpot}). We fix a divisorization $(\wi X,\pi,\wi\Sigma)$ 
of $(X,\Sigma)$ as in Definition \ref{D:divisorization}. 
Thus there exists an analytic subset $Y$ of $X$ such that $\dim Y\leq n-2$, 
$X_\sing\subset Y$, $E=\pi^{-1}(Y)$ is the final exceptional divisor, 
and $\pi:\wi X\setminus E\to X\setminus Y$ is a biholomorphism. 
We let $\wi\omega$ be a K\"ahler form on $\wi X$ such that 
$\wi\omega\geq\pi^\star\omega$ (see e.{\ke}g.\ \cite[Lemma 2.2]{CMM17}) 
and denote by $\dist$ the distance on $\wi X$ induced by $\wi\omega$. Set 
\begin{equation}\label{e:wivarphi}
\wi L:=\pi^\star L\,,\,\;\wi h_0:=\pi^\star h_0\,,\,\;\wi\alpha:=
\pi^\star\alpha=c_1(\wi L,\wi h_0)\,,\,\;\wi\varphi:=
\varphi\circ\pi\,,\,\;\wi h:=\pi^\star h=\wi h_0e^{-2\wi\varphi}\,.
\end{equation}
We write $\wi h^p=\wi h^{\otimes p}$ and $\wi h_0^p=
\wi h_0^{\otimes p}$. Corollary \ref{C:iso} implies that the map 
\begin{equation}\label{e:iso}
S\in H^0_{0,(2)}(X,L^p)\to\pi^\star S\in H^0_{0,(2)}(\wi X,\wi L^p):=
H^0_{0,(2)}(\wi X,\wi L^p,\wi\Sigma,\tau,\wi h^p,\pi^\star\omega^n)
\end{equation}
is an isometry. It follows that 
\begin{equation}\label{e:wiFS}
\wi P_p=P_p\circ\pi\,,\,\;\wi\gamma_p=\pi^\star\gamma_p
\end{equation}
are the Bergman kernel function, resp.\ Fubini-Study current, 
of the space $H^0_{0,(2)}(\wi X,\wi L^p)$. Note that  
\begin{equation}\label{e:wiFSpot}
\frac{1}{p}\,\wi\gamma_p=\wi\alpha+dd^c\wi\varphi_p\,,\,
\text{ where }\,\wi\varphi_p=\wi\varphi+\frac{1}{2p}\,\log\wi P_p\,.
\end{equation}
We call the function $\wi\varphi_p$ the {\em global Fubini-Study potential} of $\wi\gamma_p$. 

Let $\wi\varphi_\eq$ be the equilibrium envelope of 
$(\wi\alpha,\wi\Sigma,\tau,\wi\varphi)$ defined in \eqref{e:enveq},
\begin{equation}\label{e:wienveq}
\wi\varphi_\eq=\wi\varphi_{\eq,\wi\Sigma,\tau}=
\sup\{\psi:\,\psi\in\mathcal{L}(\wi X,\wi\alpha,\wi\Sigma,\tau),\;
\psi\leq\wi\varphi \,\text{ on } \wi X\}\,,
\end{equation}
where $\mathcal{L}(\wi X,\wi\alpha,\wi\Sigma,\tau)$ is defined 
in \eqref{e:calL}. Let $s_{\wi\Sigma_j}$ be the canonical section of 
$\cO_{\wi X}(\wi\Sigma_j)$ and fix a smooth Hermitian metric $g_j$ on 
$\cO_{\wi X}(\wi\Sigma_j)$ such that 
\begin{equation}\label{e:wisig}
\sigma_j:=|s_{\wi\Sigma_j}|_{g_j}<1 \text{ on } \wi X,\;1\leq j\leq\ell\,.
\end{equation}
Set 
\begin{equation}\label{e:wibts}
\beta_j=c_1(\cO_{\wi X}(\wi\Sigma_j),g_j)\,,\,\;\wi\theta=
\wi\alpha-\sum_{j=1}^\ell\tau_j\beta_j\,.
\end{equation}
Note that $[\wi\Sigma_j]=\beta_j+dd^c\log\sigma_j$, by the Lelong-Poincar\'e formula. 

In the above setting, we have the following theorem which shows 
that on $\wi X$, the global Fubini-Study potentials $\wi\varphi_p$ 
converge to the equilibrium envelope $\wi\varphi_\eq$ of $\wi\varphi$.
\begin{Theorem}\label{T:wiFSpot}
Let $X,L,\Sigma,\tau$ verify assumptions (A)-(D), 
and assume that $(L,\Sigma,\tau)$ is big and there exists a K\"ahler form 
$\omega$ on $X$. Let $h$ be a continuous Hermitian metric on $L$, 
let $\varphi$, $\wi\varphi$, $\wi\varphi_p,\wi\varphi_\eq,\wi\theta$ 
be defined in \eqref{e:varphi}, \eqref{e:wivarphi}, \eqref{e:wiFSpot}, 
\eqref{e:wienveq}, resp.\ \eqref{e:wibts}, and set 
$Z:=\wi\Sigma_1\cup\ldots\cup\wi\Sigma_\ell\cup \NAmp\big(\{\wi\theta\}\big)$. 
Then the following hold:
\\[2pt]
(i) $\wi\varphi_p\to\wi\varphi_\eq$ in $L^1(\wi X,\wi\omega^n)$ 
and locally uniformly on $\wi X\setminus Z$ as $p\to\infty$.
\\[2pt]
(ii) If $\varphi$ is H\"older continuous on $X$ 
then there exist a constant $C>0$ and $p_0\in\N$ such that 
for all $x\in\wi X\setminus Z$ and $p\geq p_0$ we have 
\[|\wi\varphi_p(x)-\wi\varphi_\eq(x)|\leq
\frac{C}{p}\,\big(\log p+\big|\log\dist(x,Z)\big|\big).\]
In particular, we have the convergence of the Fubini-Study currents 
defined in \eqref{e:wiFS},
\[\frac{1}{p}\,\wi\gamma_p=
\wi\alpha+dd^c\wi\varphi_p\to\wi T_\eq:=
\wi\alpha+dd^c\wi\varphi_\eq\,,\,\text{ as $p\to\infty$, 
weakly on $\wi X$}\,.\]
\end{Theorem}
The proof is done by estimating the partial Bergman kernel $\wi P_p$ 
from \eqref{e:wiFS} (see \cite{Ber07}, \cite{Ber09}, \cite{CM15}, \cite{RWN17} 
for similar approaches). Let 
\[\Omega_{\wi\varphi}(\delta)=\sup\big\{|\wi\varphi(x)-\wi\varphi(y)|:\,x,y\in\wi X,\;\dist(x,y)<\delta\big\}\]
be the modulus of continuity of $\wi\varphi$. We first deal with the upper estimate for $\log\wi P_p$.

\begin{Proposition}\label{P:ueBkf}
In the setting of Theorem \ref{T:wiFSpot}, there exists a constant $C>0$ such that 
\[\frac{1}{2p}\,\log\wi P_p(x)\leq\frac{C}{p}\,(1-\log\delta)+
\delta+\Omega_{\wi\varphi}(\delta)\,,\]
for all $p\geq1$, $\delta\in(0,1)$, and $x\in\wi X$ with $\dist(x,E)\geq\delta$.
\end{Proposition}

\begin{proof} By compactness, there exist constants $r_0>0$, $C_1>1$ 
with the following properties: every $x\in\wi X$ has a 
contractible Stein coordinate neighborhood $U_x$ centered at $x$ such that:

$(i)$ the ball $\overline\B(x,r_0)\subset U_x$ and the Lebesgue measure 
in coordinates satisfies $d\lambda\leq C_1\frac{\wi\omega^n}{n!}$\,;

$(ii)$ $C_1^{-1}|z-y|\leq\dist(z,y)\leq C_1|z-y|$ holds for $z,y\in \overline\B(x,r_0)$;

$(iii)$ $\wi L$ has a local holomorphic frame $e_x$ on $U_x$ such that if 
$|e_x|_{\wi h_0}=e^{-\psi_x}$ then $\psi_x$ is a Lipschitz function 
with Lipschitz constant $C_1$ on $U_x$.

Moreover, there exists $K>0$ such that 
\[\pi^\star\omega^n(x)\geq C_1^{-1}\dist(x,E)^K\,\wi\omega^n(x)\,,\,\;
\forall\,x\in\wi X\,.\]
Indeed, using local embeddings of $X$ into $\C^N$ we have that 
$\omega\gtrsim i\sum_{j=1}^Ndz_j\wedge d\overline z_j$, 
so the above claim follows from the Lojasiewicz inequality. 

We let $\delta\in(0,1)$ and fix $x\in\wi X\setminus E$ with 
$\dist(x,E)\geq\delta$. Let $r<r_0$, $r<(2C_1)^{-1}\dist(x,E)$. 
If $S\in H^0_{0,(2)}(\wi X,\wi L^p)$, $\|S\|_p=1$, 
we write $S=se_x^{\otimes p}$, where $s\in\cO_{\wi X}(U_x)$. 
Using the subaverage inequality we obtain:
\[|S(x)|_{\wi h^p}^2=|s(x)|^2e^{-2p(\psi_x(x)+
\wi\varphi(x))}\leq\frac{n!C_1}{\pi^nr^{2n}}\,
e^{2p\big(\max\limits_{\B(x,r)}(\psi_x+\wi\varphi)-\psi_x(x)-\wi\varphi(x)\big)}
\int_{\B(x,r)}|S|_{\wi h^p}^2\,\frac{\wi\omega^n}{n!}\,\cdot\]
If $y\in\B(x,r)$ then $\dist(y,E)\geq\dist(x,E)-C_1r\geq\frac{1}{2}\,
\dist(x,E)\geq\frac{\delta}{2}$, hence
\[\wi\omega^n(y)\leq C_1\dist(y,E)^{-K}\,
\pi^\star\omega^n(y)\leq2^KC_1\delta^{-K}\,\pi^\star\omega^n(y)\,.\]
Therefore 
\[|S(x)|_{\wi h^p}^2\leq\frac{C_2}{r^{2n}\delta^K}\,
e^{2p\big(\max\limits_{\B(x,r)}(\psi_x+\wi\varphi)-\psi_x(x)-
\wi\varphi(x)\big)}\int_{\B(x,r)}|S|_{\wi h^p}^2\,\frac{\pi^\star\omega^n}{n!}\,,
\]
where $C_2=n!\,\pi^{-n}2^KC_1^2$. Since $S\in H^0_{0,(2)}(\wi X,\wi L^p)$ 
is arbitrary with $\|S\|_p=1$, we infer that 
\begin{align*}\frac{1}{2p}\,\log\wi P_p(x)&\leq\frac{\log C_2}{2p}-
\frac{n}{p}\,\log r-\frac{K}{2p}\log\delta+\max_{\B(x,r)}\psi_x-\psi_x(x)+
\max\limits_{\B(x,r)}\wi\varphi-\wi\varphi(x) \\
&\leq\frac{\log C_2}{2p}-\frac{n}{p}\,\log r-\frac{K}{2p}\log\delta+
C_1r+\Omega_{\wi\varphi}(C_1r)\,.
\end{align*}
Let $M>C_1$ be a constant such that $(2M)^{-1}\dist(x,E)<r_0$ for all $x\in\wi X$. 
Taking $r=\frac{\delta}{2M}$ in the last estimate we see that the conclusion 
holds with a constant $C=C(n,C_2,K,M)$.
\end{proof}

\begin{Remark}
The conclusion of Proposition \ref{P:ueBkf} holds 
in fact for the full Bergman kernel of the space 
$H^0_{(2)}(\wi X,\wi L^p,\wi h^p,\pi^\star\omega^n)$.
\end{Remark}

\begin{Proposition}\label{P:uewiFS}
In the setting of Theorem \ref{T:wiFSpot}, there exists a constant $C>0$ 
such that for all $p\geq1$ and $\delta\in(0,1)$ the following estimate holds on $\wi X$:
\[\wi\varphi_p\leq\wi\varphi_\eq+C\Big(\delta+\frac{1}{p}-
\frac{\log\delta}{p}\Big)+2\Omega_{\wi\varphi}(C\delta)\,.\]
\end{Proposition}

\begin{proof} If $p\geq1$ and $0<\delta<1$, we have by Proposition \ref{P:ueBkf} that 
\begin{equation}\label{e:Fp}
F_p(\delta):=\sup\Big\{\frac{1}{2p}\,\log\wi P_p(x):\,x\in\wi X,\;
\dist(x,E)\geq\delta\Big\}\leq\frac{C}{p}\,(1-\log\delta)+\delta+\Omega_{\wi\varphi}(\delta)\,.
\end{equation}

Let $E_\delta:=\{x\in\wi X:\,\dist(x,E)<\delta\}$ and fix $x_0\in E$. 
There exist a coordinate neighborhood $U_{x_0}$ centered at $x_0$ 
and a constant $C_1=C_{1,x_0}>1$ with the following properties: 

$(i)$ the polydisc $\overline\Delta^n(0,2)\subset U_{x_0}$ 
and $C_1^{-1}|z-y|\leq\dist(z,y)\leq C_1|z-y|$ for 
$z,y\in\overline\Delta^n(0,2)$, where $|z|:=\max_{1\leq j\leq n}|z_j|$.

$(ii)$ $\wi\alpha=dd^c\rho$ on $U_{x_0}$, where $\rho$ is a smooth function 
with Lipschitz constant $C_1$ on  $\overline\Delta^n(0,2)$ 
(with respect to the norm $|z|$ from $(i)$);

$(iii)$ $E\cap\Delta^n(0,2)=\Big(\bigcup_{j=1}^k\{z_j=0\}\Big)\cap\Delta^n(0,2)$, 
for some $1\leq k\leq n$. \\
Note that $(iii)$ can be achieved since $E$ is a divisor with only normal crossings. Hence  
\[E_{\delta/C_1}\cap\Delta^n(0,1)\subset\Big\{z\in\Delta^n(0,1):\,
\min_{1\leq j\leq k}|z_j|<\delta\Big\}\,.\]
Let $z\in E_{\delta/C_1}\cap\Delta^n(0,1)$ and without loss of generality assume that
\[|z_1|<\delta,\,\ldots,\,|z_l|<\delta,\;|z_{l+1}|\geq\delta,\,\ldots,\,|z_k|\geq\delta\,.\] 
The function  $v:=\rho+\wi\varphi_p=\rho+\wi\varphi+\frac{1}{2p}\,\log\wi P_p$ 
is psh on $U_{x_0}$. By the maximum principle applied on 
$V_l:=\{\zeta=(\zeta_1,\ldots,\zeta_l,z_{l+1},\ldots,z_n):
\,|\zeta_j|\leq\delta,\;1\leq j\leq l\}$ it follows that 
\[v(z)\leq\max_{\partial_dV_l}v\leq\max_{\partial_dV_l}\rho+
\max_{\partial_dV_l}\wi\varphi+F_p(\delta/C_1)\,,\] 
as $|\zeta_1|=\ldots=|\zeta_l|=\delta$ for $\zeta$ in the distinguished boundary 
$\partial_dV_l$, hence $\dist(\zeta,E)\geq\delta/C_1$. Since  
\[\max_{\partial_dV_l}\rho\leq\rho(z)+2C_1\delta\,,\,\;
\max_{\partial_dV_l}\wi\varphi\leq\wi\varphi(z)+\Omega_{\wi\varphi}(2C_1\delta)\,,\]
we conclude that 
\[\wi\varphi_p(z)\leq\wi\varphi(z)+2C_1\delta+\Omega_{\wi\varphi}(2C_1\delta)+
F_p(\delta/C_1)\,,\,\;\forall\,z\in E_{\delta/C_1}\cap\Delta^n(0,1)\,.\]

Using a finite cover of $E$ with neighborhoods $\Delta^n(0,1)\subset U_{x_0}$, 
$x_0\in E$, we infer by above that there exists a constant $C'>1$ 
such that for all $p\geq1$ and $0<\delta<1$ we have 
\begin{equation}\label{e:wivarphi1}
\wi\varphi_p(x)\leq\wi\varphi(x)+C'\delta+\Omega_{\wi\varphi}(C'\delta)+F_p(\delta/C')\
\end{equation}
for all $x\in E_{\delta/C'}$. Note that $\wi\varphi_p(x)=
\wi\varphi(x)+\frac{1}{2p}\,\log\wi P_p\leq\wi\varphi(x)+
F_p(\delta/C')$ for $x\in\wi X\setminus E_{\delta/C'}$, 
hence the estimate \eqref{e:wivarphi1} holds for all $x\in\wi X$, $p\geq1$ 
and $0<\delta<1$. Since $\wi P_p$ is the Bergman kernel function of the space 
$H^0_0(\wi X,\wi L^p,\wi\Sigma,\tau)$, it follows that $\wi\varphi_p$ 
has Lelong number $\geq t_{j,p}/p\geq\tau_j$ 
along $\wi\Sigma_j$, $1\leq j\leq\ell$. Therefore, by \eqref{e:wienveq} and \eqref{e:Fp},
\[\wi\varphi_p\leq\wi\varphi_\eq+C'\delta+
\Omega_{\wi\varphi}(C'\delta)+F_p(\delta/C')\leq\wi\varphi_\eq+
(C'+1)\delta+2\Omega_{\wi\varphi}(C'\delta)+
\frac{C}{p}\,\Big(1-\log\frac{\delta}{C'}\Big),\]
which concludes the proof.
\end{proof}

We next deal with the lower bound for $\log\wi P_p$ and $\wi\varphi_p$. 
The following form of the $L^2${\ke}-{\ke}estimates 
of H\"ormander/Andreotti-Vesentini 
for $\db$ is due to Demailly \cite[Th\'eor\`eme 5.1]{D82} 
(see also \cite[Theorem 5.2]{CM15}, \cite[Theorem 2.5]{CMM17}):

\begin{Theorem}\label{T:dbar}
Let $M$ be a complete K\"ahler manifold of dimension $n$, 
$\Omega$ be a (not necessarily complete) K\"ahler form on $M$, 
$\chi$ be a qpsh function on $M$, and $(F,h)$ be a singular 
Hermitian holomorphic line bundle on $M$. Assume that there exist 
constants $A,B>0$ such that 
\[\ric\Omega\geq-2\pi B\Omega\,,\,\;dd^c\chi\geq
-A\Omega\,,\,\;c_1(F,h)\geq(1+B+A/2)\Omega\,.\]
If $g\in L_{0,1}^2(M,F,{\rm loc})$ satisfies $\db g=0$ 
and $\int_M|g|^2_he^{-\chi}\,\Omega^n<+\infty$, 
then there exists $u\in L_{0,0}^2(M,F,{\rm loc})$ with 
$\db u=g$ and $\int_M|u|^2_he^{-\chi}\,\Omega^n\leq
\int_M|g|^2_he^{-\chi}\,\Omega^n$.
\end{Theorem}

Since $(L,\Sigma,\tau)$ is big and $\wi X$ is K\"ahler, 
it follows from Theorem \ref{T:main1} that the class 
$\{\wi\theta\}=\{\wi\theta\}_\ddbar$ is big, where $\wi\theta$ 
is defined in \eqref{e:wibts}. By \cite{D92} and \cite[Theorem 3.17]{Bo04}, 
there exists a $\wi\theta$-psh function $\eta$ with analytic singularities on $\wi X$ such that  
\[\{\eta=-\infty\}=\NAmp\big(\{\wi\theta\}\big)\,,\,\;\eta\leq
-1\,,\,\text{ and }\, \wi\theta+
dd^c\eta\geq\varepsilon_0\wi\omega\geq\varepsilon_0\pi^\star\omega\]
hold on $\wi X$, for some constant $\varepsilon_0>0$. 

\begin{Proposition}\label{P:lewiFS}
In the setting of Theorem \ref{T:wiFSpot}, 
there exist a constant $C>0$ and $p_0\in\N$ such that for all 
$p\geq p_0$ the following estimate holds on $\wi X\setminus Z$:
\[\wi\varphi_p\geq\wi\varphi_\eq+\frac{C}{p}\,\eta+
\frac{1}{p}\,\sum_{j=1}^\ell\log\sigma_j\,.\]
\end{Proposition}

\begin{proof} We consider the Bergman space 
$H^0_{(2)}(\wi X,\wi L^p,H_p,\wi\omega^n)$ of $L^2$\ke-{\ke}integrable sections 
of $\wi L^p$ with respect to the volume form $\wi\omega^n$ on $\wi X$ 
and the metric $H_p:=\wi h_0^pe^{-2\psi_p}$ on $\wi L^p$, where 
\[\psi_p=(p-p_0)\wi\varphi_\eq+p_0\eta+\sum_{j=1}^\ell(p_0\tau_j+1)\log\sigma_j\,,\]
and $p_0\in\N$ will be specified later. Since $\sigma_j<1$, $\eta<0$, 
and $\wi\varphi_\eq\leq\wi\varphi$, we have for $p\geq p_0$ that $\psi_p\leq(p-p_0)\wi\varphi$. 
Moreover, 
\begin{align*}
c_1(\wi L^p,H_p)&=p\wi\alpha+(p-p_0)dd^c\wi\varphi_\eq+
p_0dd^c\eta+\sum_{j=1}^\ell(p_0\tau_j+1)([\wi\Sigma_j]-\beta_j) \\
&=(p-p_0)(\wi\alpha+dd^c\wi\varphi_\eq)+p_0(\wi\theta+dd^c\eta)+
\sum_{j=1}^\ell(p_0\tau_j+1)[\wi\Sigma_j]-\sum_{j=1}^\ell\beta_j \\
&\geq(p_0\varepsilon_0-C_1)\wi\omega\,,
\end{align*}
where $C_1>0$ is a constant such that 
$\sum_{j=1}^\ell\beta_j\leq C_1\wi\omega$. 
If $p_0$ is chosen large enough (i.{\ke}e.\ $p_0\varepsilon_0-C_1\geq1+B+A/2$, 
where $A,B$ are as in Theorem \ref{T:dbar}) and $p\geq p_0$, we use 
Theorem \ref{T:dbar} for $\wi L^p$, with suitable weights $\chi$ as in the proof of 
\cite[Theorem 5.1]{CM15}, to show that there exists a constant $C_2>0$ 
such that for all $p\geq p_0$ and $x\in\wi X\setminus Z$ there exists 
$S_x\in H^0_{(2)}(\wi X,\wi L^p,H_p,\wi\omega^n)$ with $S_x(x)\neq0$ and 
\[\|S_x\|^2_{H_p,\wi\omega^n}\leq C_2|S_x(x)|^2_{H_p}\,.\]
Note that $H_p=\wi h^pe^{2F_p}$, where 
$F_p=p\wi\varphi-\psi_p\geq p_0\wi\varphi$. 
Let $a:=\min_{\wi X}\wi\varphi$. Then $F_p\geq ap_0$ 
and since 
$\wi\omega\geq\pi^\star\omega$ we obtain
\[\|S_x\|^2_{H_p,\wi\omega^n}=
\int_{\wi X}|S_x|^2_{\wi h^p}e^{2F_p}\,\wi\omega^n\geq 
e^{2ap_0}\|S_x\|^2_{\wi h^p,\pi^\star\omega^n}\,.\]
By \eqref{e:eqreq}, we have that $\wi\varphi_\eq=
\wi\varphi_\req+\sum_{j=1}^\ell\tau_j\log\sigma_j$ on $\wi X$, 
where $\wi\varphi_\req=\wi\varphi_{\req,\wi\Sigma,\tau}$ 
is the reduced equilibrium envelope of $(\wi\alpha,\wi\Sigma,\tau,\wi\varphi)$ 
defined in \eqref{e:envreq}. 
Hence 
\[\psi_p=(p-p_0)\wi\varphi_\req+p_0\eta+\sum_{j=1}^\ell(p\tau_j+1)\log\sigma_j\]
is a qpsh function with Lelong numbers $\geq p\tau_j+1$ 
along $\wi\Sigma_j$. Since $\|S_x\|^2_{H_p,\wi\omega^n}<+\infty$, this shows that 
$\ord(S_x,\wi\Sigma_j)\geq\lfloor\tau_jp\rfloor+1\geq t_{j,p}$, 
so $S_x\in H^0_0(\wi X,\wi L^p,\wi\Sigma,\tau)$. 
Moreover, since
\[e^{2ap_0}\|S_x\|^2_{\wi h^p,\pi^\star\omega^n}\leq 
C_2|S_x(x)|^2_{\wi h^p}e^{2F_p(x)}\,,\]
we infer that 
\[\wi P_p(x)\geq C_2^{-1}e^{2ap_0-2F_p(x)}.\]
Note that on $\wi X$ we have, for some constant $C_3>0$,  
\[F_p=p\wi\varphi-\psi_p=p(\wi\varphi-\wi\varphi_\eq)+
p_0\wi\varphi_\req-p_0\eta-\sum_{j=1}^\ell\log\sigma_j\leq 
p(\wi\varphi-\wi\varphi_\eq)+C_3-p_0\eta-
\sum_{j=1}^\ell\log\sigma_j.
\]
It follows that there exists a constant $C_4>0$ 
such that for all $p\geq p_0$,
\begin{equation}\label{e:BK2}
\frac{1}{2p}\,\log\wi P_p\geq\wi\varphi_\eq-
\wi\varphi-\frac{C_4}{p}+\frac{p_0}{p}\,\eta+
\frac{1}{p}\,\sum_{j=1}^\ell\log\sigma_j
\end{equation}
holds on $X\setminus Z$. Using \eqref{e:wiFSpot} 
and since $\eta\leq-1$, \eqref{e:BK2} implies that 
\[\wi\varphi_p\geq\wi\varphi_\eq+
\frac{C_4+p_0}{p}\,\eta+\frac{1}{p}\,\sum_{j=1}^\ell\log\sigma_j\]
holds on $X\setminus Z$ for all $p\geq p_0$.
\end{proof}
\begin{proof}[Proof of Theorem \ref{T:wiFSpot}] Since $\eta$, $\log\sigma_j$ 
are qpsh functions on $\wi X$ with analytic singularities along 
$\NAmp\big(\{\wi\theta\}\big)$, resp.\ $\wi\Sigma_j$, 
there exist constants $N_j,M_j>0$, $0\leq j\leq\ell$, such that 
\[\eta(x)\geq-N_0\big|\log\dist\big(x,\NAmp\big(\{\wi\theta\}\big)\big)\big|-
M_0\,,\,\;\log\sigma_j(x)\geq-N_j\big|\log\dist\big(x,\wi\Sigma_j\big)\big|-M_j\,,\]
for all $x\in\wi X$. Together with Proposition \ref{P:lewiFS}, 
these imply that there exist a constant $C_1>0$ and $p_0\in\N$ such that if $p\geq p_0$ then 
\begin{equation}\label{e:lewiFS1}
\wi\varphi_p\geq\wi\varphi_\eq-\frac{C_1}{p}\,\Big(\big|\log\dist\big(x,Z\big)\big|+1\Big)
\end{equation}
holds on $\wi X$. 

Since $\varphi$ is continuous we have that $\wi\varphi$ 
is continuous. Let $\varepsilon>0$ and fix $\delta=
\delta(\varepsilon) $ such that $C\delta+
2\Omega_{\wi\varphi}(C\delta)<\frac{\varepsilon}{2}\,$, 
where $C$ is the constant from Proposition \ref{P:uewiFS}. 
There exists $p_\varepsilon$ such that 
$\frac{C}{p}\,(1-\log\delta)<\frac{\varepsilon}{2}$ for 
$p\geq p_\varepsilon$. Hence by Proposition \ref{P:uewiFS},
\begin{equation}\label{e:lewiFS2}
\wi\varphi_p\leq\wi\varphi_\eq+C\Big(\delta+\frac{1}{p}-
\frac{\log\delta}{p}\Big)+2\Omega_{\wi\varphi}(C\delta)\leq
\wi\varphi_\eq+\varepsilon 
\end{equation}
holds on $\wi X$ for $p\geq p_\varepsilon$. Note that 
$\log\dist\big(\LargerCdot,Z\big)\in L^1(\wi X,\wi\omega^n)$ 
(see e.{\ke}g.\ \cite[Lemma 5.2]{CMN16} and its proof). 
Assertion $(i)$ of Theorem \ref{T:wiFSpot} now follows from 
\eqref{e:lewiFS1} and \eqref{e:lewiFS2}.

Assume next that the function  $\varphi$ is H\"older continuous on $X$. 
Then $\wi\varphi$ is H\"older continuous on $\wi X$, so $\Omega_{\wi\varphi}(\delta)\leq 
C_2\delta^\nu$ for some constant $C_2>0$, where $\nu$ 
is the H\"older exponent of $\wi\varphi$. 
Taking $\delta=p^{-1/\nu}$ in Proposition \ref{P:uewiFS} 
we see that there exists a constant $C_3>0$ such that 
\begin{equation}\label{e:lewiFS3}
\wi\varphi_p\leq\wi\varphi_\eq+C\Big(p^{-1/\nu}+
p^{-1}+\frac{\log p}{\nu p}\Big)+2C_2\,\frac{C^\nu}{p}\leq
\wi\varphi_\eq+C_3\,\frac{\log p}{p}
\end{equation}
holds on $\wi X$ for $p\geq2$. Assertion $(ii)$ follows 
immediately from \eqref{e:lewiFS1} and \eqref{e:lewiFS3}.
\end{proof}

Using Theorem \ref{T:wiFSpot} we can prove the convergence 
of the Fubini-Study currents and their global potentials on $X$.

\begin{proof}[Proof of Theorem \ref{T:FSpot}] 
Let $(\wi X,\pi,\wi\Sigma)$ be a divisorization of $(X,\Sigma)$ 
as in Definition \ref{D:divisorization}. Then $\pi:\wi X\setminus E\to X\setminus Y$ 
is a biholomorphism, where $Y\supset X_\sing$ is an analytic subset of $X$ with 
$\dim Y\leq n-2$ and $E=\pi^{-1}(Y)$. 
By Theorem \ref{T:wiFSpot},  
$\wi\varphi_p=\varphi_p\circ\pi\to\wi\varphi_\eq$ in $L^1(\wi X,\wi\omega^n)$, 
where $\wi\omega$ is a K\"ahler form on $\wi X$ such that $\wi\omega\geq\pi^\star\omega$. 
Recall that the functions $\wi\varphi_p,\wi\varphi_\eq$ are $\wi\alpha$-psh on $\wi X$, 
where $\wi\alpha=\pi^\star\alpha$. We define 
$\varphi_\eq:=\wi\varphi_\eq\circ\pi^{-1}$ on $X\setminus Y\subset X_\reg$. Then 
\begin{equation}\label{e:Hspeed}
\int_{X\setminus Y}|\varphi_p-\varphi_\eq|\,\omega^n=
\int_{\wi X\setminus E}|\wi\varphi_p-\wi\varphi_\eq|\,\pi^\star\omega^n\leq
\int_{\wi X}|\wi\varphi_p-\wi\varphi_\eq|\,\wi\omega^n\to 0 \text{ as } p\to\infty\,.
\end{equation}
Since $\pi^\star(\alpha+dd^c\varphi_\eq)=\wi\alpha+dd^c\wi\varphi_\eq\geq0$, 
it follows that $\varphi_\eq$ is $\alpha$-psh on $X\setminus Y$. 

It remains to show that $\varphi_\eq$ extends to an $\alpha$-psh function on $X$. 
Let $x_0\in Y$ and $U_{x_0}$ be a neighborhood of $x_0$ in $X$ 
on which $L$ has a local holomorphic frame $e_{x_0}$, 
and let $|e_{x_0}|_{h_0}=e^{-\rho}$, where $\rho$ is a smooth function on $U_{x_0}$. 
Then $dd^c\rho=\alpha$. Since $\rho\circ\pi+\wi\varphi_\eq$ is psh on 
$\pi^{-1}(U_{x_0})\setminus E$, we infer that $v:=\rho+\varphi_\eq$ is psh on 
$U_{x_0}\setminus Y$. Hence $v$ extends to a psh function on $U_{x_0}$ since $X$ 
is normal and $\dim Y\leq n-2$ \cite[Satz 4]{GR56}. Therefore $\varphi_\eq$ 
extends to an $\alpha$-psh function on $X$.

The second assertion of Theorem \ref{T:FSpot} 
follows at once from \eqref{e:Hspeed} and Theorem \ref{T:wiFSpot} \emph{(ii)}, 
since the function $\log\dist\big(\LargerCdot,Z\big)\in L^1(\wi X,\wi\omega^n)$.
\end{proof}
We record here an immediate consequence of
Theorems \ref{T:FSpot} and \ref{T:main2} for the case when $\Sigma=\emptyset$.
\begin{Corollary}\label{cor:FSpot}
Let $(X,\omega)$ be an irreducible compact normal K\"ahler space
of dimension $n$ and let $L$ be a big line bundle on $X$ endowed with
a continuous Hermitian metric $h$. We denote by $P_p$
the Bergman kernel function of $H^0(X,L^p)$ relative to $h^p$
and $\omega^n/n!$ and by $\gamma_p$ the corresponding Fubini-Study current
\eqref{e:FSpot}. Let $h_0$ be a smooth metric on $L$ and denote by
$\varphi_p$ the global Fubini-Study potential relative to $h_0$ \eqref{e:FSpot}. 
Then the following assertions hold:

(a) There exists an 
$\alpha$-psh function $\varphi_\eq$ on $X$ such that as $p\to\infty$ we have
$\varphi_p\to\varphi_\eq$ in $L^1(X,\omega^n)$, 
$\frac{1}{p}\,\gamma_p\to T_\eq:=
\alpha+dd^c\varphi_\eq$ and 
$\displaystyle\frac{1}{p}\,[s_p=0]\to T_\eq$\,, 
in the weak sense of currents on $X$, 
for $\sigma_\infty$-{\ke}a.{\ke}e.\ $\{s_p\}_{p\geq1}\in\X_\infty$\,.

(b) If, in addition, $h$ is H\"older continuous then there exist a constant 
$C>0$ and $p_0\in\N$ such that 
$\|\varphi_p-\varphi_\eq\|_{L^1(X,\omega^n)}\leq C\,(\log p)/p$
for all $p\geq p_0$ and the large deviation principle of 
Theorem \ref{T:main2} (ii) holds.
\end{Corollary}
Note that if $X$ is smooth, then 
$\varphi_\eq=\sup\{\psi\in\PSH(X,\alpha):\,\psi\leq\varphi\text{ on }X \}$
is the usual
upper envelope.
\section{Proof of the equisdistribution Theorem \ref{T:main2}}\label{S:Tmain2}

Let $h,h_0,\varphi$ be as in \eqref{e:varphi}. 
Let $P_p,\gamma_p$ be the Bergman kernel function and 
Fubini-Study current of the space $H^0_{0,(2)}(X,L^p)$, and let $\varphi_p$ 
be the global Fubini-Study potential of $\gamma_p$ (see \eqref{e:FSpot}). 
We start by proving that zero divisors of random sections distribute 
like the Fubini-Study currents.

\begin{Theorem}\label{T:speed}
Let $X,L,\Sigma,\tau$ verify assumptions (A)-(D), let $h$ be a bounded singular 
Hermitian metric on $L$, and assume that $(L,\Sigma,\tau)$ is big and there exists a 
K\"ahler form $\omega$ on $X$. Then there exists a constant $c>0$ 
with the following property: For any sequence of positive numbers 
$\{\lambda_p\}_{p\geq1}$ such that 
\[\liminf_{p\to\infty} \frac{\lambda_p}{\log p}>(1+n)c\,,\]
there exist subsets $E_p\subset\X_p$ such that 

(a) $\sigma_p(E_{p})\leq cp^n\exp(-\lambda_p/c)$ holds for all $p$ sufficiently large;

(b) if $s_p\in\X_p\setminus E_p$ we have 
\[\Big|\frac{1}{p}\,\langle[s_p=0]-\gamma_p,\phi\rangle\Big|\leq
\frac{c\lambda_p}{p}\,\| \phi\|_{\cC^2}\,,\] 
for any $(n-1,n-1)$-form $\phi$ of class $\cC^2$ on $X$. 

In particular, the last estimate holds for 
$\sigma_\infty$-{\ke}a.{\ke}e.\ $\{s_p\}_{p\geq1}\in\X_\infty$ 
provided that $p$ is large enough.
\end{Theorem} 

\begin{proof} We apply the Dinh-Sibony equidistribution theorem 
for meromorphic transforms \cite[Theorem 4.1]{DS06}, 
as in the proof of \cite[Theorem 4.2]{CMN16}. 
Our present situation is easier as we only deal with currents of bidegree $(1,1)$. 
We fix a divisorization $(\wi X,\pi,\wi\Sigma)$ of $(X,\Sigma)$ as in 
Definition \ref{D:divisorization}, and let $\wi\omega$ be a K\"ahler form on $\wi X$. 
Let $\wi L,\wi h$ be as in \eqref{e:wivarphi}, and $\wi P_p,\wi\gamma_p$ 
be the Bergman kernel function and Fubini-Study current of the space 
$H^0_{0,(2)}(\wi X,\wi L^p)$ (see \eqref{e:wiFS}). 
By \eqref{e:iso}, $H^0_{0,(2)}(\wi X,\wi L^p),H^0_{0,(2)}(X,L^p)$ are isometric. 
We proceed in two steps. 

{\em Step 1.} We prove here that Theorem \ref{T:speed} 
holds for the spaces $H^0_{0,(2)}(\wi X,\wi L^p)$. Set 
\[\wi\X_p:=\P H^0_{0,(2)}(\wi X,\wi L^p),\;\sigma_p=
\omega_\FS^{d_p},\;(\wi\X_\infty,\sigma_\infty):= 
\prod_{p=1}^\infty (\wi\X_p,\sigma_p),\text{ where } d_p=\dim\wi\X_p=\dim\X_p.\] 
We proceed as in \cite[Section 4]{CMN16} and consider 
the Kodaira maps as meromorphic transforms of codimension $n-1$, 
$\Phi_p:\wi X\dashrightarrow\P H^0_{0,(2)}(\wi X,\wi L^p)$, with graph 
\[\Gamma_p=\big\{(x,\tilde s)\in\wi X\times\P H^0_{0,(2)}(\wi X,\wi L^p):\,
\tilde s(x)=0\big\}\,.\]
If $\delta_{\wi s_p}$ is the Dirac mass at 
$\tilde s_p\in\P H^0_{0,(2)}(\wi X,\wi L^p)$ then $\Phi_p^\star(\delta_{\tilde s_p})$ 
is well defined for generic $\tilde s_p$ and $\Phi_p^\star(\delta_{\tilde s_p})=[\tilde s_p=0]$. 
Moreover, by \cite[Lemma 4.5]{CMN16} (see also \cite{SZ99}) we have
\[\langle\Phi_p^\star(\sigma_p),\phi\rangle=
\int_{\wi\X_p}\langle[\tilde s_p=0],\phi\rangle\,d\sigma_p(\tilde s_p)=
\langle\wi\gamma_p,\phi\rangle\,,\]
where $\phi$ is a smooth $(n-1,n-1)$ form on $\wi X$. 
The intermediate degrees of $\Phi_p$ are
\[d(\Phi_p):=\int_{\wi X}\Phi_p^\star(\sigma_p)\wedge\wi\omega^{n-1}=
p\int_{\wi X}c_1(\wi L,\wi h)\wedge\wi\omega^{n-1},\;
\delta(\Phi_p):= \int_{\wi X}\Phi_p^\star(\omega_\FS^{d_p-1})\wedge\wi\omega^n=
\int_{\wi X}\wi\omega^n.\]
For $\varepsilon>0$ let
\[\wi E_p(\varepsilon):=\bigcup_{\|\phi\|_{\cC^2}\leq 1}\big\lbrace
\tilde s\in 
\wi\X_p:\,\big|\big\langle[\tilde s=0]-\wi\gamma_p\,,\phi\big\rangle\big|\geq 
d(\Phi_p)\varepsilon \big\rbrace\,.\]
By \cite[Lemma 4.2 (d)]{DS06}, we infer, using the estimates of 
\cite[Lemma 4.6]{CMN16} for a projective space, that there exist constants 
$C_1,a_1,M_1>0$ such that 
\[\sigma_p(\wi E_p(\varepsilon))\leq 
C_1d_pe^{-a_1\varepsilon p+M_1\log d_p}\,,\,\;\forall\,\varepsilon>0,\,p\geq1\,.\]
By Siegel's lemma $d_p=O(p^n)$, so $\sigma_p(\wi E_p(\varepsilon))\leq 
C_2p^ne^{-a_1\varepsilon p+C_2\log p}$, for some constant $C_2>0$. 
We let $\varepsilon_p:=\lambda_p/p$ and $\wi E_p:=\wi E_p(\varepsilon_p)$. 
If $\liminf_{p\to\infty} \frac{\lambda_p}{\log p}>2C_2/a_1$, 
it follows that $\sigma_p(\wi E_p)\leq C_2p^ne^{-a_1\lambda_p/2}$ 
for all $p$ sufficiently large. Set 
\[c=\max\left(\frac{2C_2}{a_1(1+n)}\,,\frac{2}{a_1}\,,
C_2,\int_{\wi X}c_1(\wi L,\wi h)\wedge\wi\omega^{n-1}\right).\]
If $\liminf_{p\to\infty} \frac{\lambda_p}{\log p}>(1+n)c$ then 
$\sigma_p(\wi E_p)\leq cp^ne^{-\lambda_p/c}$ for all $p$ large enough. Moreover,  
\[\Big|\frac{1}{p}\,\langle[\tilde s_p=0]-\gamma_p,\phi\rangle\Big|\leq
\frac{d(\Phi_p)}{p}\,\frac{\lambda_p}{p}\,
\| \phi\|_{\cC^2}\leq\frac{c\lambda_p}{p}\,\| \phi\|_{\cC^2}\] 
holds for any $\tilde s_p\in\wi\X_p\setminus\wi E_p$ 
and any $(n-1,n-1)$-form $\phi$ of class $\cC^2$ on $\wi X$. 

{\em Step 2.} We complete now the proof of the theorem when $X$ is singular. 
Let $c>0$ be the constant constructed in Step 1, let $\{\lambda_p\}_{p\geq1}$ 
be a sequence of positive numbers such that 
$\liminf_{p\to\infty} \frac{\lambda_p}{\log p}>(1+n)c$, and let 
$\wi E_p\in\wi\X_p$ be the corresponding sets constructed in Step 1. Recall by \eqref{e:iso} 
that the map $S\in H^0_{0,(2)}(X,L^p)\to\pi^\star S\in H^0_{0,(2)}(\wi X,\wi L^p)$ 
is an isometry and let 
$E_p=F_p(\wi E_p)$, where $F_p:\wi\X_p\to\X_p$ is the isomorphism induced by this isometry. 
Then $\sigma_p(E_p)\leq cp^ne^{-\lambda_p/c}$ for all $p$ sufficiently large. 
Note that $\pi_\star\wi\gamma_p=\gamma_p$ and, if $s\in\X_p$, 
then $\pi_\star[\tilde s=0]=[s=0]$, 
where $s=F_p(\tilde s)$. Hence if $s_p\in\X_p\setminus E_p$ and 
$\phi$ is any $(n-1,n-1)$-form of class $\cC^2$ on $X$, we have 
\[\Big|\frac{1}{p}\,\langle[s_p=0]-\gamma_p,\phi\rangle\Big|=
\Big|\frac{1}{p}\,\langle[\tilde s_p=0]-\wi\gamma_p,\pi^\star\phi\rangle\Big|
\leq\frac{c\lambda_p}{p}\,\|\pi^\star\phi\|_{\cC^2}
\leq\frac{cc_1\lambda_p}{p}\,\|\phi\|_{\cC^2}\,,\]
for some constant $c_1>0$. The last assertion of Theorem \ref{T:speed} 
follows as in \cite[Theorem 4.2]{CMN16}. 
\end{proof}

\begin{proof}[Proof of Theorem \ref{T:main2}] 

If $h$ is continuous then by Theorem \ref{T:FSpot}, 
$\frac{1}{p}\,\gamma_p\to T_\eq$ weakly on $X$, as $p\to\infty$. 
Moreover, by Theorem \ref{T:speed}, 
$\frac{1}{p}\,\big([s_p=0]-\gamma_p\big)\to 0$ weakly on $X$, 
for $\sigma_\infty$-{\ke}a.{\ke}e.\ $\{s_p\}_{p\geq1}\in\X_\infty$. 
This proves assertion $(i)$.

Assume now that $h$ is H\"older continuous. There exists a constant $C'>0$ such that 
\[-C'\|\phi\|_{\cC^2}\,\omega^n\leq dd^c\phi\leq C'\|\phi\|_{\cC^2}\,\omega^n\,,\]
for every real valued $(n-1,n-1)$-form $\phi$ of class $\cC^2$ on $X$. 
Hence the total variation of $dd^c\phi$ satisfies  
$|dd^c\phi|\leq C'\|\phi\|_{\cC^2}\,\omega^n$ 
(see e.{\ke}g.\ \cite{BCM}). Using Theorem \ref{T:FSpot} we infer that 
\[\Big|\Big\langle\frac{1}{p}\,\gamma_p-T_\eq\,,\phi\Big\rangle\Big|=
\Big|\int_X(\varphi_p-\varphi_\eq)\,dd^c\phi\Big|\leq 
C'\|\phi\|_{\cC^2}\int_X|\varphi_p-\varphi_\eq|\,\omega^n
\leq CC'\|\phi\|_{\cC^2}\,\frac{\log p}{p}\,,\]
for all $p\geq p_0$ and $\phi$ as above. 
Assertion $(ii)$ follows combining this and Theorem \ref{T:speed}. 
\end{proof}


\begin{thebibliography}{XXXXX}

\bibitem[BCM]{BCM} T.\ Bayraktar, D.\ Coman and G.\ Marinescu, 
{\em Universality results for zeros of random holomorphic sections}, 
{\tt arXiv:1709.10346}, Trans.\ Amer.\ Math.\ Soc., to appear. 

\bibitem[BCHM]{BCHM} T.\ Bayraktar, D.\ Coman, H.\ Herrmann and G.\ Marinescu,
\emph{A survey on zeros of random holomorphic sections}, 
Dolomites Res.\ Notes Approx.\ \textbf{11} (2018), 
Special Issue Norm Levenberg, 1--19. 

\bibitem[Be1]{Ber07} R.\ Berman, 
{\em Bergman kernels and equilibrium measures for ample line bundles}, 
preprint (2007), {\tt  arXiv:0704.1640v1}. 

\bibitem[Be2]{Ber09} R.\ Berman, 
{\em Bergman kernels and equilibrium measures for line bundles over projective manifolds},  
Amer.\ J.\ Math.\ {\bf 131} (2009), 1485--1524.
 
\bibitem[BD]{BD12} R.\ Berman and J.-P.\  Demailly,  
{\em  Regularity of plurisubharmonic upper envelopes in big cohomology classes}, 
Perspectives in analysis, geometry, and topology, 39--66, Progr. Math., {\bf 296}, 
Birkh\"auser/Springer, New York, 2012.


\bibitem[BM]{BM97} E.\ Bierstone and P.\ Milman, 
{\em Canonical desingularization in characteristic zero by blowing 
up the maximum strata of a local invariant}, 
Invent.\ Math.\ {\bf 128} (1997), 207--302.

\bibitem[BP]{BlPo31}
A.\ Bloch and G.\ P\'olya, {\em On the roots of certain algebraic equations}, 
Proc.\ London Math.\ Soc.\ \textbf{33} (1931), 102--114. 

\bibitem[B1]{Bl05}
T. Bloom, \emph{Random polynomials and Green functions}, 
Int.\ Math.\ Res.\ Not.\ \textbf{2005}, no.\ 28, 1689--1708.

\bibitem[B2]{Bl09}
T. Bloom, \emph{Weighted polynomials and weighted pluripotential theory}, 
Trans.\ Amer.\ Math.\ Soc. \textbf{361}
(2009), no.\ 4, 2163--2179.

\bibitem[BL]{BL15} T.\ Bloom and N.\ Levenberg, 
{\em Random polynomials and pluripotential-theoretic extremal functions}, 
Potential Anal.\ {\bf 42} (2015), 311--334.


\bibitem[Bon]{Bon93} L. Bonavero, {\em In\'egalit\'es de Morse holomorphes singuli\`eres}, 
J.\ Geom.\ Anal.\ {\bf 8} (1998), 409--425; 
announced in C.\ R.\ Acad.\ Sci.\ Paris S\'er. I Math.\ {\bf 317} (1993), 1163--1166.


\bibitem[Bou]{Bo04}  S. Boucksom, 
{\em Divisorial Zariski decompositions on compact complex manifolds},  
Ann.\ Sci.\ \'E‰c.\ Norm.\ Sup.\ (4) {\bf 37} (2004), 45--76.

\bibitem[Ca]{Ca99} D. Catlin, 
{\em The Bergman kernel and a theorem of Tian}, 
in {\em Analysis and geometry in several complex variables 
(Katata, 1997)}, 1--23, Trends Math., Birkh\"auser, Boston, 1999. 

\bibitem[CM1]{CM15} D.\ Coman and G.\ Marinescu, 
{\em Equidistribution results for singular metrics on line bundles}, 
Ann.\ Sci.\ \'Ec.\ Norm.\ Sup\'er.\ (4) \textbf{48} (2015), 497--536.
  
\bibitem[CM2]{CM13} D.  Coman and G.  Marinescu, 
{\em Convergence of Fubini-study currents for orbifold line bundles}, 
Internat. J. Math. {\bf 24} (2013), no. 7, 1350051, 27 pp. 
 
\bibitem[CM3]{CM17} D. Coman and G. Marinescu,  
{\em On the first order asymptotics of partial Bergman kernels}, 
Ann. Fac. Sci. Toulouse Math. (6) {\bf 26} (2017), 1193--1210. 
   
\bibitem[CMM]{CMM17}  D. Coman,  X. Ma and G.  Marinescu, 
{\em Equidistribution for sequences of line bundles on normal K\"ahler spaces.}  
Geom. Topol. {\bf 21} (2017), 923--962.
   
\bibitem[CMN1]{CMN16} D.\ Coman, G.\ Marinescu and V.-A.\ Nguy\^en,
{\em H\"older singular metrics on big line bundles and equidistribution},  
Int.\ Math.\ Res.\ Notices {\bf 2016}, no. 16, 5048--5075.

\bibitem[CMN2]{CMN18} D.\ Coman, G.\ Marinescu and V.-A.\ Nguy\^en,
{\em Approximation and  equidistribution results for pseudo-effective line bundles},  
 J.\ Math.\ Pures Appl.\ (9) {\bf 115} (2018), 218--236.
 
\bibitem[Da]{Dar17} T. Darvas, {\em Weak geodesic rays in the space of 
K\"ahler potentials and the class $\mathcal E(X,\omega)$}, 
J. Inst. Math. Jussieu {\bf 16} (2017), 837--858. 

\bibitem[D1]{D82} J.\ P.\ Demailly, {\em Estimations $L^2$ pour 
l'op\'erateur $\overline\partial$ d'un fibr\'e holomorphe semipositif 
au--dessus d'une vari\'et\'e k\"ahl\'erienne compl\`ete}, 
Ann.\ Sci.\ \'Ecole Norm.\ Sup.\ (4) {\bf 15} (1982), no.\ 3, 457--511.

\bibitem[D2]{D85} J.-P.\ Demailly, {\em Mesures de Monge-€"Amp\`ere 
et caract\'erisation g\'eom\'etrique des vari\'et\'es alg\'ebriques affines}, 
M\'em. Soc. Math. France {\bf 19}, Soc. Math. France, Paris (1985).

\bibitem[D3]{D90} J.-P.\ Demailly, 
{\em Singular Hermitian metrics on positive line bundles}, 
in {\em  Complex algebraic varieties (Bayreuth, 1990)}, 
Lecture Notes in Math.\ 1507, Springer, Berlin, 1992, 87--104.

\bibitem[D4]{D92} J.-P.\ Demailly, 
{\em Regularization of closed positive currents and intersection theory}, 
J.\ Algebraic Geom.\ {\bf 1} (1992), 361--409.

\bibitem[D5]{D93} J.-P.\ Demailly, {\em Monge-Amp\`ere operators, Lelong 
numbers and intersection theory}, in {\em Complex analysis and geometry}, 
Plenum, New York, 1993, 115--193.


\bibitem[D6]{D94} J.-P. Demailly, {\em  
Regularization of closed positive currents of type $(1,1)$ by the flow of a Chern connection}, in  
{\em  Contributions to complex analysis and analytic geometry},  
105--126, Aspects Math., {\bf E26}, {\it Vieweg, Braunschweig,} 1994.


\bibitem[DMM]{DMM16}  T.-C.\ Dinh, X. Ma and G. Marinescu, 
{\em Equidistribution and  convergence  speed for zeros of holomorphic sections 
of singular Hermitian line bundles}, J.\ Funct.\ Anal.\ \textbf{271} (2016), no.\ 11, 3082--3110.

\bibitem[DMN]{DMN17}  T.-C.\ Dinh, X. Ma and  V.-A.\ Nguyen, {\em Equidistribution   
speed for Fekete points  associated with an  ample line bundle}, 
Ann.\ Sci.\ \'Ec.\ Norm.\ Sup\'er.\ (4)  50, (2017), 545--578.

\bibitem[DMS]{DMS} T.-C.\ Dinh, G.\ Marinescu and V.\ Schmidt, 
{\em Asymptotic distribution of zeros of holomorphic sections in the non-compact setting}, 
J.\ Stat.\ Phys.\ {\bf 148} (2012), no.\ 1, 113--136.
  
\bibitem[DS]{DS06} T. C. Dinh and N. Sibony, {\em Distribution des valeurs de transformations 
m\'eromorphes et applications}, Comment.\ Math.\ Helv.\ {\bf 81} (2006), no.\ 1, 221--258.

\bibitem[ET]{ET50} 
P.\ Erd\H{o}s and P. Tur\'an, {\em On the distribution of roots of polynomials}, 
Ann. Math. \textbf{57} (1950), 105--119.
 
 

\bibitem[G]{Gr62} H.\ Grauert, 
\emph{{\"Uber Modifikationen und exzeptionelle analytische Mengen}}, 
Math.\ Ann.\ \textbf{146} (1962), 331--368.

\bibitem[GR1]{GR56} H.\ Grauert and R.\ Remmert, 
{\em Plurisubharmonische Funktionen in komplexen R\"aumen}, 
Math.\ Z.\ {\bf 65} (1956), 175--194.

\bibitem[GR2]{GR84} H.\ Grauert and R.\ Remmert, 
{\em Coherent Analytic Sheaves}, Springer, Berlin, 1984. 
Grundlehren der Mathematischen Wissenschaften, 265, 
Springer-Verlag, Berlin, 249 pp., 1984. 

 \bibitem[H]{Ham56}
J.\ M.\ Hammersley, 
{\em The zeros of a random polynomial},
Proc. 3rd Berkeley Sympos. Math. Statist. Probability \textbf{2} (1956), 89--111. 


\bibitem[JS]{JS93}S.\ Ji and B.\ Shiffman, 
\emph{Properties of compact complex manifolds carrying 
closed positive currents}, 
J.\ Geom.\ Anal.\ \textbf{3} (1993), no.\ 1, 37--61.

\bibitem[Ka]{Kac49}
M.\ Kac, {\em On the average number of real roots of a random algebraic equation, II}, 
Proc.\ London Math.\ Soc.\ \textbf{50} (1948), 390--408.
 
\bibitem[Ki1]{Kiselman78} C.\ O.\ Kiselman,  
\emph{The partial Legendre transformation for plurisubharmonic functions}, 
Invent.\ Math.\ {\bf 49} (1978), no.\ 2, 137--148. 

\bibitem[Ki2]{Kiselman94} C.\ O.\ Kiselman,  
\emph{Attenuating the singularities of plurisubharmonic functions},
Ann.\ Polon.\ Math.\ {\bf 60} (1994), no.\ 2, 173--197.
  
  
\bibitem[KP]{KP02}S.\ Krantz and H.\ Parks,
\emph{A primer of real analytic functions}, 
Second edition. Birkh\"auser Advanced Texts: Basler Lehrb\"ucher. 
Birkh\"auser Boston, Inc., Boston, MA, 2002. xiv+205 pp.
  
\bibitem[LS]{LS99} F.\ L\'arusson and R.\ Sigurdsson, 
{\em Plurisubharmonic extremal functions, Lelong numbers and coherent ideal sheaves}, 
Indiana Univ.\ Math.\ J.\ {\bf 48} (1999), 1513--1534. 

\bibitem[MM]{MM07} X.\ Ma and G.\ Marinescu, 
{\em Holomorphic Morse Inequalities and Bergman Kernels}, 
Progress in Mathematics, 254. Birkh\"auser Verlag, Basel, 2007. xiv+422 pp.


  
\bibitem[O]{Ohs87} T.\ Ohsawa, {\em Hodge spectral sequence and 
symmetry on compact K\"ahler spaces,} Publ.
Res.\ Inst.\ Math.\ Sci.\ {\bf 23} (1987) 613--625.
  
\bibitem[PS]{PS14}
F.\ Pokorny and M.\ Singer, {\em Toric partial density functions and stability of toric varieties,} 
Math.\ Ann.\ {\bf 358} (2014), no.\ 3-4, 879--923.

\bibitem[RaS]{RS05} A.\ Rashkovskii and R.\ Sigurdsson, 
{\em Green functions with singularities along complex spaces}, 
Internat.\ J.\ Math.\ {\bf 16} (2005), 333--355.
 
\bibitem[RoS]{RS13}
J. Ross and M. Singer, {\em  Asymptotics of partial density functions for divisors}, 
{\tt arXiv:1312.1145,}
(2013).

\bibitem[RT06]{RT06} J. Ross and R. Thomas, 
{\em An obstruction to the existence of constant scalar curvature K\"ahler metrics}, 
J.\ Differential Geom.\  {\bf  72} (2006), 429--466.

\bibitem[RWN1]{RWN14} J. Ross and D. Witt Nystr\"om, 
{\em  Analytic test configurations and geodesic rays,} J.\ Symplectic Geom. 
{\bf 12} (2014), 125--169.

\bibitem[RWN2]{RWN17} J. Ross and D. Witt Nystr\"om, 
{\em Envelopes of positive metrics with prescribed singularities}, 
Ann.\ Fac.\ Sci.\ Toulouse Math.\ (6) {\bf 26} (2017), 687--728.
 
\bibitem[R]{Ru98} W. D. Ruan, 
{\em Canonical coordinates and Bergmann metrics}, 
Comm. Anal. Geom. {\bf 6} (1998), 589--631.
  
\bibitem[SZ]{SZ99} B.\ Shiffman and S.\ Zelditch, 
{\em Distribution of zeros of random and quantum chaotic sections of positive line bundles}, 
Comm.\ Math.\ Phys.\ {\bf 200} (1999), 661--683.
  

\bibitem[T]{Ti90} G.\ Tian, 
{\em On a set of polarized K\"ahler metrics on algebraic manifolds}, 
J.\ Differential Geom.\ {\bf 32} (1990), 99--130.

\bibitem[Z1]{Z98} S.\ Zelditch, {\em Szeg\"o kernels and a theorem of Tian}, 
Internat.\ Math.\ Res.\ Notices {\bf 1998}, no.\ 6, 317--331.

\bibitem[Z2]{Z18} S.\ Zelditch, 
\emph{Quantum ergodic sequences and equilibrium measures}, 
Constr.\ Approx.\ \textbf{47} (2018), no.~1, 89--118. 

\bibitem[ZZ]{ZeZh}
S.\ Zelditch and P.\ Zhou,
\emph{Interface asymptotics of partial Bergman kernels on $S^1$-symmetric 
K\"ahler manifolds}, arXiv:1604.06655.

\end{thebibliography}
\end{document}